%% file: mainnew.tex
\documentclass[a4paper]{amsart}
\usepackage[latin1]{inputenc}
\usepackage{amssymb}
\usepackage{amsmath}
\usepackage{mathrsfs}
\usepackage{eufrak}
\usepackage{amsthm}
\usepackage{thmtools}
\usepackage{amsfonts}
\usepackage{textcomp}
\usepackage{graphicx}
\usepackage[pdftex]{color}
\usepackage[dvipsnames]{xcolor}
\usepackage{paralist}
\usepackage[shortlabels]{enumitem}
\usepackage{hyperref}
\hypersetup{hypertexnames=false}
\usepackage{comment}
\usepackage{tikz}
\usepackage{tikz-cd}
\usepackage[arrow, matrix, curve]{xy}
\usepackage{quiver}
\usepackage{soul}
\usepackage{cleveref}
\usetikzlibrary{decorations.pathreplacing,calligraphy, calc, arrows.meta, shapes.misc, decorations.pathmorphing}

\newtheorem*{corollary*}{Corollary}
\newtheorem{theorem}{Theorem}[section]
\newtheorem*{theorem*}{Theorem}

\newtheorem{corollary}[theorem]{Corollary}
\newtheorem{Question M}[theorem]{Question M}
\newtheorem{Conjecture S}[theorem]{Conjecture S}
\newtheorem{Conjecture K}[theorem]{Conjecture K}
\newtheorem{lemma}[theorem]{Lemma}

\newtheorem{proposition}[theorem]{Proposition}
\newtheorem{problem}[theorem]{Problem}

\newtheorem*{claim*}{Claim}
\newtheorem*{conjecture}{Conjecture}

\theoremstyle{definition}
\newtheorem{definition}[theorem]{Definition}

\newtheorem*{theorem }{Theorem}
\newtheorem{remark}[theorem]{Remark}
\newtheorem{example}[theorem]{Example}

\theoremstyle{remark}

\numberwithin{equation}{theorem}

\makeatletter
\renewcommand*\env@matrix[1][\
arraystretch]{%
  \edef\arraystretch{#1}%
  \hskip -\arraycolsep
  \let\@ifnextchar\new@ifnextchar
  \array{*\c@MaxMatrixCols c}}
\makeatother

\renewcommand{\mod}{\operatorname{mod}}

\newcommand{\Ext}{\operatorname{Ext}}

\newcommand{\del}{\operatorname{del}}
\newcommand{\im}{\operatorname{Im}}

\newcommand{\id}{\operatorname{id}}
\newcommand{\Hom}{\operatorname{Hom}}

\newcommand{\grade}{\operatorname{grade}}
\newcommand{\cograde}{\operatorname{cograde}}

\renewcommand{\top}{\operatorname{\mathrm{top}}}

\newcommand{\soc}{\operatorname{\mathrm{soc}}}

\renewcommand{\mod}{\operatorname{mod}}

\renewcommand{\ker}{\mathrm{Ker}}
\newcommand{\coker}{\mathrm{Coker}}

\newcommand{\idim}{\operatorname{idim}}
\newcommand{\pdim}{\operatorname{pdim}}

\title{Classification of Auslander-Gorenstein monomial algebras: The acyclic case}
\date{\today}
\dedicatory{Dedicated to Claus Michael Ringel on the occasion of his eightieth birthday}

\author[V.~Kl\'asz]{Vikt\'oria Kl\'asz$^\star$}%
\address[V.~Kl\'asz]{Mathematical Institute of the University of Bonn, Endenicher Allee 60, 53115 Bonn, Germany}%
\email{klasz@math.uni-bonn.de}%
\thanks{$^\star$Supported by the Deutsche Forschungsgemeinschaft
(DFG, German Research Foundation) under Germany's Excellence Strategy - GZ 2047/1, Projekt-ID
390685813\\}

\author[M.~Kleinau]{Markus Kleinau$^\star$}%
\address[M.~Kleinau]{Mathematical Institute of the University of Bonn, Endenicher Allee 60, 53115 Bonn, Germany}%
\email{mkleinau@math.uni-bonn.de}%

\author[R.~Marczinzik]{Ren\'e Marczinzik}%
\address[R.~Marczinzik]{Mathematical Institute of the University of Bonn, Endenicher Allee 60, 53115 Bonn, Germany}
\email{marczire@math.uni-bonn.de}
\date{\today}

\begin{document}

\begin{abstract}
We give a linear algebraic classification of Auslander regular acyclic monomial algebras via the Bruhat factorisation of the Coxeter matrix. Namely, we show under mild assumptions that a monomial acyclic quiver algebra is Auslander regular if and only if its Coxeter matrix $C$ has a Bruhat factorisation $U_1 P U_2$ with $U_1$ the identity matrix. In particular, this holds without restrictions for linear Nakayama algebras and we use the Bruhat decomposition to answer a question raised by Ringel by showing that his homological permutation coincides with the permutation coming from the Bruhat factorisation of the Coxeter matrix. 
We also use our methods to show that general Auslander regular acyclic quiver algebras are echelon-independent, proving a conjecture of Defant-Jiang-Marczinzik-Segovia-Speyer-Thomas-Williams and we answer another question by Ringel on the delooping level of simple modules over Nakayama algebras.
\end{abstract}
\maketitle

\setcounter{tocdepth}{1}
\tableofcontents

\section{Introduction}
\label{sec::intro}
Let $R$ be a Noetherian (not necessarily commutative) ring. Then $R$ is called 
\textcolor{blue}{\emph{n-Gorenstein}} if there exists an injective coresolution of the $R$-module $R_R$
$$0 \rightarrow R_R \rightarrow I^0 \rightarrow I^1 \rightarrow I^2 \rightarrow \ldots$$
such that the flat dimension of $I^i$ is at most $i$ for all $n > i \geq 0$.
We call $R$ \textcolor{blue}{\emph{Auslander-Gorenstein}} if the above resolution is finite and if $R$ is $n$-Gorenstein for all $n\ge 0.$ 
We say that $R$ is \textcolor{blue}{\emph{Auslander regular}} if it is Auslander-Gorenstein of finite global dimension. 
By Bass's classical result \cite{B}, a commutative local ring $R$ is Gorenstein (resp. regular) in the classical sense of commutative algebra if and only if it is Auslander-Gorenstein (resp. Auslander regular). This was one of the motivations of Auslander to introduce the notion of Auslander-Gorenstein rings.
There are many important classes of non-commutative rings that are Auslander-Gorenstein.
Let us mention, for example, the enveloping algebras of finite dimensional Lie algebras, Weyl algebras, and rings of $\mathbb{C}$-linear differential operators on an irreducible smooth subvariety of an affine
space, see \cite[Chapter 3]{VO} for more details.
In this work, we focus on the setting of finite dimensional algebras, where recently several important classes of algebras were discovered to be Auslander-Gorenstein.
In \cite{IM}, it was shown that incidence algebras of finite lattices are Auslander-regular if and only if the lattice is distributive, and in \cite{KMM} it was proven that blocks of category $\mathcal{O}$ for semisimple Lie algebras are Auslander regular.
In the following, we will assume that our rings are finite dimensional $K$-algebras over a field $K$, and that they are given as path algebras of connected quivers modulo an admissible ideal. 
We call such algebras quiver algebras. 
This is no loss of generality over an algebraically closed field $K$, as any finite dimensional ring-indecomposable $K$-algebra $A$ is Morita equivalent to a quiver algebra, and all our homological notions in this article are invariant under Morita equivalence.
Recall that the \textcolor{blue}{\emph{Cartan matrix} $\omega_A$}=$(\omega_{ij})_{i,j}$  of a quiver algebra $A$ with $n$ primitive orthogonal idempotents $e_i$, corresponding to the vertices of the quiver, is defined as the $n \times n$-matrix with entries $\omega_{ij}=\dim e_j A e_i$ for $1 \leq i,j \leq n$. 
In other words, the $i$-th row vector of $\omega_A$ is the dimension vector of the indecomposable injective $A$-module $I(i)\cong D(Ae_i),$ or equivalently, the $i$-th column vector of $\omega_A$ is the dimension vector of the indecomposable projective $A$-module $P(i)\cong e_iA.$  
When $A$ has finite global dimension, the Cartan matrix is invertible over the integers. (For a proof of the previous statements, see  for  example \cite[Chapter III, Proposition 3.8 and 3.10]{ASS}.) Therefore, in this case, we can define the \textcolor{blue}{\emph{Coxeter matrix} $C_A$} of $A$ as the $n \times n$-matrix $C_A:=- \omega_A^T \omega_A^{-1}.$ 
The Coxeter matrix is the linear algebraic version of the Auslander-Reiten translate for hereditary algebras and the main object of study in the spectral analysis of finite dimensional algebras, see for example the survey \cite{LP}.
Recently, it was shown in \cite{KMT} that the Coxeter matrix plays a central role in the study of Auslander regular algebras as well.
In this article, we are interested in the following problem:
\begin{problem} \label{monomialARclassificationproblem}
Classify all Auslander regular monomial algebras.
\end{problem}
A quiver algebra $A=KQ/I$ is called \textcolor{blue}{\emph{monomial}} if the admissible ideal $I$ is generated by monomials, i.e. by a set of paths. This class of quiver algebras contains many important and well-studied classes of algebras, such as hereditary algebras, gentle algebras, and Nakayama algebras. In this paper, 
Nakayama algebras play a central role. Recall that a quiver algebra $A$ is called a \textcolor{blue}{\emph{Nakayama algebra}} if every indecomposable left and right $A$-module is uniserial, meaning it has a unique composition series.
Nakayama algebras can be grouped into two classes: \textcolor{blue}{\textit{linear}} and \textcolor{blue}{\textit{cyclic}} Nakayama algebras.
The connected linear Nakayama algebras are those whose quiver is of the form

\begin{equation}
\label{eq:linear_quiver}
    \begin{tikzcd}
	1 & 2 & \ldots & {n-1} & n
	\arrow[from=1-1, to=1-2]
	\arrow[from=1-2, to=1-3]
	\arrow[from=1-3, to=1-4]
	\arrow[from=1-4, to=1-5]
\end{tikzcd}
\end{equation}
and the cyclic Nakayama algebras are those with a cyclic quiver

\begin{equation}
    \label{eq:cyclic_quiver}
\begin{tikzcd}
	& 1 & 2 \\
	n &&& 3 \\
	{n-1} &&& m \\
	& {m+2} & {m+1}
	\arrow[from=1-2, to=1-3]
	\arrow[from=1-3, to=2-4]
	\arrow[from=2-1, to=1-2]
	\arrow[dotted, from=2-4, to=3-4]
	\arrow[from=3-1, to=2-1]
	\arrow[from=3-4, to=4-3]
	\arrow[dotted, from=4-2, to=3-1]
	\arrow[from=4-3, to=4-2]
\end{tikzcd}
\end{equation}

Note that this is a much more restrictive class than monomial algebras, where $Q$ is allowed to be arbitrary.
However, as was shown in  \cite{Kla}, the classification of Auslander regular monomial algebras can be reduced to classifying all Auslander regular Nakayama algebras, we refer to \cite[Theorem 4.1 and 5.1]{Kla} for full details. In particular, a full classification by quiver and relations for 2-Gorenstein monomial algebras is achieved there, which we also discuss in Section \ref{subsec::monomial_results}.
In conclusion, Problem \ref{monomialARclassificationproblem} reduces to the following.

\begin{problem} \label{NakayamaARclassificationproblem}
Classify all Auslander regular Nakayama algebras.
\end{problem}

In this article, we solve Problem \ref{NakayamaARclassificationproblem}, and hence Problem \ref{monomialARclassificationproblem} for acyclic quivers by reducing the problem to linear algebra.
Recall that for an invertible matrix $M$ a \textcolor{blue}{\emph{Bruhat decomposition}} of $M$ is a decomposition $M=U_1 P U_2$ with upper triangular matrices $U_1$ and $U_2$ and a permutation matrix $P$. The Bruhat decomposition is a classical tool in mathematics, especially important in Lie theory, we refer for example to \cite{Bu}. While the Bruhat decomposition is not unique in general, the permutation matrix $P$ is uniquely determined by $M$.
Let $A$ be a finite dimensional quiver algebra of finite global dimension with a fixed numbering of the vertices of $A$ from $1$ to $n$. The permutation corresponding to the permutation matrix in a Bruhat decomposition of the Coxeter matrix of $A$ is called the \textcolor{blue}{\emph{Coxeter permutation}} of $A$ with respect to the given order of the vertices.
In \cite{KMT}, the Coxeter permutation was introduced and it was shown that it generalises the well-studied rowmotion bijection on distributive lattices.
In \cite{DJMSSTW}, the Coxeter permutation (or more precisely, its inverse) was studied under the name echelonmotion for general incidence algebras of posets and it was shown that it recovers several well-studied bijections in combinatorics.
Note that, in general, the Coxeter permutation depends on the numbering of the vertices of our quiver. 
For linear Nakayama algebras, we will always assume in the following that the vertices are labelled by the canonical ordering from 1 to $n$, as shown in (\ref{eq:linear_quiver}).
More generally, for acyclic monomial algebras, we will always assume that their vertices are \textcolor{blue}{\textit{naturally labelled}}, which means that if there is a path from a vertex $i$ to a vertex $j$ with $j\neq i$ then we have $i<j.$
Our first main result is the following:
\begin{theorem}(Theorem \ref{thm::C=PU_iff_ARbijnak}  and \Cref{thm::AG_iff_C=PU_acyclic_2Gorenstein_monomial})
\label{thm:linearNakayama_AGiffC=PU_sec1}
Let $A$ be a linear Nakayama algebra or a 2-Gorenstein, acyclic monomial algebra. Then $A$ is Auslander regular if and only if there exists a Bruhat decomposition $C_A=U_1 P U_2$ of the Coxeter matrix of $A$ where $U_1$ is the identity matrix.
\end{theorem}

We mention that this theorem is not only interesting from a theoretical point of view but also from a computational aspect. When testing whether a given acyclic monomial algebra with around 50 simple modules is Auslander regular using the homological definition, the GAP-package \cite{QPA} took several hours, while it only takes seconds when using the above theorem that reduces the problem to linear algebra. 
We remark that it is an open question whether the same characterisation of Auslander regular algebras via their Coxeter matrix holds for incidence algebras of finite bounded posets, see \cite[Question 5.3]{KMT} where such a result was proven for the first time for the subclass of lattices.
In \cite[Conjecture 8.2]{DJMSSTW} it was conjectured that if the incidence algebra $KP$ of a poset $P$ is Auslander regular, then it is echelon-independent  in the sense that the Coxeter permutation does not depend on the natural labelling, see \cite[Definition 1.1]{DJMSSTW} for the precise definition.
More generally, we define an acyclic quiver algebra to be \textcolor{blue}{\emph{Coxeter-independent}} if the Coxeter permutation does not depend on the chosen natural labelling of the algebra. 

 The next theorem proves one direction of Question 5.3 in \cite{KMT}.

\begin{theorem}
(Theorem \ref{maintheoremauslanderregimpliescoxbruhat})
Let $A$ be a naturally labelled acyclic Auslander regular algebra.
Then there exists a Bruhat decomposition of the Coxeter matrix $C_A=U_1 P U_2$ with $U_1=\id$.
In this case, the Coxeter permutation coincides with the Auslander-Reiten permutation.
\end{theorem}

We refer to the preliminaries for the definition of the Auslander-Reiten permutation associated to an Auslander-Gorenstein algebra. As a consequence of the previous theorem, we prove the conjecture \cite[Conjecture 8.2]{DJMSSTW} in the following much more general version:

\begin{theorem}(Theorem \ref{cor::coxeter-indep})
Every acyclic Auslander regular algebra is Coxeter-independent. 
\end{theorem}

We will study Coxeter-independent algebras more generally for acyclic algebras in forthcoming work.
In \cite{R}, Ringel studied the finitistic dimension of Nakayama algebras, and during his study, he discovered a \textcolor{blue}{\emph{homological bijection} $h$} on the simple modules of Nakayama algebras that is defined as follows:
For a simple module $S$ define $\textcolor{blue}{e(S)}:=\min \{ \pdim S, \pdim I(S) \}$ and set $h(S)=\top \Omega^{e(s)}(N(S)),$ 
where $N(S)=S$ if $e(S)$ is odd and $N(S)=I(S)$ if $e(S)$ is even. 
Here $I(S)$ denotes the injective envelope of $S.$
Note that $e(S)$ was proven to be finite for every simple module $S$ over a Nakayama algebra.
Assume that our Nakayama algebra has vertices $\{1,\ldots, n\}.$
Then we define the map $\textcolor{blue}{\hat{h}}: \{1,\ldots,n \} \rightarrow \{1,\ldots,n\}$ as $\hat{h}(i)=j$ if $h(S_i)=S_j$, where $S_i$ denotes the simple module corresponding to vertex $i$ in the quiver. 
We call $\hat{h}$ the \textcolor{blue}{\emph{homological permutation}} of the Nakayama algebra.
The importance of the invariant $e(S)$ comes from the fact that $\max \{ e(S) \mid S \ \text{simple} \}$ is equal to the finitistic dimension of a Nakayama algebra and $h$ can be used to show that the left finitistic dimension coincides with the right finitistic dimension for Nakayama algebras, as was shown by Ringel in \cite{R}. We also refer to \cite{R} for more information on the homological bijection of Nakayama algebras and its properties.
After the discovery of the homological bijection, Ringel posed the question to several experts whether this bijection also has an elementary description. This is especially interesting in the case of linear Nakayama algebras, as they are in canonical bijection with Dyck paths (see, for example \cite{MRS} and \cite{KMMRS}), and thus, in this case, Ringel's homological bijection can be seen as a map from the set of Dyck paths to permutations.
Our second main result provides an elementary description of $h$:
\begin{theorem}(Theorem {\ref{thm::h=Coxeter_perm}})
\label{thm:h=Coxeter_perm_sec1}
Let $A$ be a linear Nakayama algebra. Then the Coxeter permutation of $A$ coincides with Ringel's homological permutation.
\end{theorem}

In \cite{G}, G\'elinas introduced the \textcolor{blue}{\emph{delooping level}} of a module $M$ as $\textcolor{blue}{\del (M)}:= \inf \{ t \geq 0 \mid \Omega^t(M) \ \text{is a (t+1)-th syzygy module} \}$. The delooping level of an algebra $A$ is then defined as the maximum of the delooping levels of the simple $A$-modules, and one of the central results of \cite{G} is that the delooping level of an algebra can be used to give an upper bound of the finitistic injective dimension of the algebra.
Thus, it is important to understand the delooping level of simple modules.
In \cite[Remark 3.9]{R}, Ringel raised the question whether we have $e(S)=\del (S)$ for every simple $A$-module $S$ in a Nakayama algebra.
The methods in our article allow us to give a positive answer to this question.
\begin{theorem}(Theorem \ref{thm::del(S)=e(S)})
\label{thm:e(S)=del(S)_sec1}
Let $A$ be a Nakayama algebra with a simple module $S$. Then $e(S)=\del (S).$
\end{theorem}

We emphasise that \Cref{thm:e(S)=del(S)_sec1} holds for every Nakayama algebra, not only for the linear ones.

In Theorems \ref{thm:linearNakayama_AGiffC=PU_sec1} and \ref{thm:h=Coxeter_perm_sec1}, we only dealt with linear Nakayama algebras, which always have finite global dimension; thus, only the Auslander regular condition was relevant to us.
An important feature of general Auslander-Gorenstein algebras is that they come with a nice homological bijection on their simple modules, called \textcolor{blue}{\emph{Iyama's grade bijection $\phi$}}, introduced in \cite{I}.
We remark that in \cite{KMT} it was shown that Iyama's grade bijection coincides with the bijection studied by Auslander and Reiten in \cite{AR} on the index set of the simple modules, which has a slightly simpler description. We refer to the preliminaries for details.
\begin{theorem} (Theorem \ref{RingelsPerm=ARPerm})
Let $A$ be an Auslander-Gorenstein Nakayama algebra. Then Ringel's homological bijection coincides with Iyama's grade bijection.
\end{theorem}
We remark that while our main result gives a linear algebraic classification of acyclic Auslander regular monomial algebras, an explicit combinatorial enumeration is not yet known, even in the case of linear Nakayama algebras. In forthcoming work, we will report on some progress regarding this problem.
\section{Preliminaries}
We always assume that our algebras are finite dimensional over a field $K$. Modules are finitely generated right modules unless otherwise stated. The functor $D=\Hom_K(-,K)$ denotes the natural duality of a finite dimensional algebra. 
If $A$ is a finite dimensional algebra with $n$ simple modules, then $\{e_1,\ldots,e_n\}$ denotes a complete set of primitive, pairwise orthogonal idempotents of $A$. 
When $A$ is a quiver algebra, the vertices of the underlying quiver are labelled by the set $\{1,\ldots,n\}.$
The indecomposable projective, indecomposable injective, and simple $A$-modules corresponding to $e_i$ are denoted by $P(i)\cong e_iA$, $I(i)\cong D(Ae_i)$, and $S_i \cong \top P(i)\cong \soc I(i)$, respectively. 
We refer, for example, to \cite{ASS} and \cite{ARS} for introductions to the representation theory of finite dimensional algebras. 

We first phrase the definition of the Auslander-Gorenstein property in the finite dimensional setting, and then recall an important result about Auslander-Gorenstein algebras that we will need in the following.

\begin{definition}
    \label{def::AG}
    Let $A$ be a finite dimensional $K$-algebra and consider a minimal injective coresolution of the regular representation $A_A:$
    $$0 \rightarrow A_A \rightarrow I^0 \rightarrow I^1 \rightarrow I^2 \rightarrow \cdots.$$
    Then $A$ is called 
    \begin{enumerate}[\rm (i)]
        \item \textcolor{blue}{\textit{n-Gorenstein}} for an $n\ge 1$ if $\pdim I^i\le i$ for all $0\le i < n.$ 
        \item \textcolor{blue}{\textit{Auslander-Gorenstein}} if $\idim (A_A)<\infty$ and $A$ is $n$-Gorenstein for all $n\ge 1.$
        \item \textcolor{blue}{\textit{Auslander regular}} if it is Auslander-Gorenstein and has finite global dimension. 
    \end{enumerate}
\end{definition}

\begin{theorem}\cite[Theorem 3.3]{KMT}\label{AGcondition}
Let $A$ be a finite dimensional $K$-algebra. Then the following are equivalent:
\begin{enumerate}[\rm (i)]
    \item $A$ is Auslander-Gorenstein.
    \item For every simple $A$-module $S$ the following holds: $\grade(S)=\pdim I(S)<\infty.$
\end{enumerate}
\end{theorem}

Here, $\grade (M)=\inf \{i \geq 0 \mid \Ext_A^i(M,A) \neq 0 \}$ is the \textcolor{blue}{\textit{grade}} of a module $M$. 
The following proposition describes an easy way to calculate Ext-groups involving simple modules. For a proof, see \cite[Corollary 2.5.4]{Ben}.
\begin{lemma} \label{Bensonlemma}
Let $A$ be a finite dimensional algebra, $N$ be an indecomposable $A$-module and $S$ a simple $A$-module. Let $(P_i)$ be a minimal projective resolution of $N$ and $(I_i)$ a minimal injective coresolution of $N$. 
\begin{enumerate}[\rm (1)]
\item For $\ell \geq 0$, $\Ext_A^{\ell}(N,S) \neq 0$ iff $S$ is a quotient of $P_{\ell}$.  \item For $\ell \geq 0$, $\Ext_A^{\ell}(S,N) \neq 0$ iff $S$ is a submodule of $I_{\ell}$. \end{enumerate} 
\end{lemma}

Next, we recall the definition of the delooping level of a module $M,$ which was introduced by G\'elinas in \cite{G}. The \textcolor{blue}{\textit{delooping level}} of a module $M$ is defined as $\textcolor{blue}{\del (M)}:= \inf \{ t \geq 0 \mid \Omega^t(M) \ \text{is a (t+1)-th syzygy module} \}$, and the delooping level of an algebra $A$ as the supremum of the delooping levels of the simple $A$-modules. 
The following result will be useful for the proof of \Cref{thm::del(S)=e(S)}.

\begin{lemma}\cite[Proposition 2.2 (G\'elinas)]{R}
\label{lem:dell_lowerbound}
    Let $A$ be a finite dimensional algebra and $M$ a finitely generated $A$-module. If there is a (not necessarily finitely generated) module $N$ with finite injective dimension $d\ge 1$ such that $\Ext^d(M, N )\neq 0$, then $d \le \del (M)$.
\end{lemma}

\subsection{Permutations on simple modules}\label{subsec::permutations_on_simples}
In the following, we recall various permutations introduced on the simple modules of finite dimensional algebras. 
For this, we recall the classical Bruhat decomposition of an invertible matrix. 
For a proof and an algorithm, see, for example, \cite[Chapter 4]{AB} and \cite{OOV}.
\begin{theorem}
Let $M$ be an invertible real $n \times n$ matrix. Then there exists a factorisation $M=U_1PU_2$, where $P$ is a permutation matrix, and $U_1$ and $U_2$ are upper triangular matrices. 
The permutation matrix $P$ is uniquely determined by $M$.
\end{theorem}
This decomposition is called a \textcolor{blue}{\emph{Bruhat decomposition}} of $M$. 
Since the permutation matrix $P$ is uniquely determined by $M$, it makes sense to denote it by \textcolor{blue}{$P[M]$} in the following.
In general, uniqueness for $U_1$ and $U_2$ is far from being true. 
However, by assuming some additional conditions on $U_1$ and $U_2$, one can obtain a unique form of the Bruhat decomposition, we refer to \cite{OOV} for details.
Given an $n \times n$ permutation matrix $P$, we define the \textcolor{blue}{\textit{row-permutation $p_r$}} associated to $P$ as the permutation acting on $\{1,\ldots,n\}$ given by $p_r(i)=j$ if in the $i$-th row of $P$ the unique non-zero entry appears in column $j$. Dually, we define the \textcolor{blue}{\textit{column-permutation $p_c$}} associated to $P$ as the permutation acting on $\{1,\ldots,n\}$ given by $p_c(i)=j$ if in the $i$-th column of $P$ the unique non-zero entry is in row $j$. Note that $p_c$ is the inverse permutation 
of $p_r.$

Following \cite{KMT}, for an algebra $A$ of finite global dimension, we define the \textcolor{blue}{\textit{Coxeter permutation} $p_A$} as the column-permutation associated to the permutation matrix $P[C_A]$ from a Bruhat factorisation of the Coxeter matrix $C_A$. 
Note that when $A$ is an acyclic quiver algebra, then we can order its simple modules so that the Cartan matrix of $A$ is lower triangular. This happens exactly when the algebra is naturally labelled, as introduced in Sections \ref{sec::intro} and \ref{sec::auslander_regularity_and_Coxeter_matrix}. We will usually assume this for acyclic quiver algebras.
The following lemma is elementary:
\begin{lemma} \label{bruhatfromcartan}
Let $A$ be an algebra of finite global dimension with a lower triangular Cartan matrix $\omega_A$ and a Bruhat decomposition $\omega_A=\hat{U}_1 R \hat{U}_2$.
Then the Coxeter permutation $p_A$ coincides with the row-permutation associated to $R$.
\end{lemma}
\begin{proof}
If the Coxeter matrix $C_A$ has Bruhat decomposition
$C_A=- \omega_A^T \omega_A^{-1}=U_1 P U_2$, then we obtain a Bruhat decomposition of $\omega_A^{-1}$ as
$\omega_A^{-1}=(-(\omega_A^T)^{-1} U_1) P U_2$ because $(-(\omega_A^T)^{-1} U_1)$ is also upper triangular.
Thus, a Bruhat factorisation of $\omega_A$ is given by 
$U_2^{-1} P^{-1} (-(\omega_A^T)^{-1} U_1)^{-1}$, which yields the result since for a permutation matrix $P$ we have $P^T=P^{-1}$ and since the column-permutation of $P$ coincides with the row-permutation of $P^T.$
\end{proof}

The next proposition provides explicit formulas for the row- and column-permutations associated to $P[M]$, for an invertible $n\times n$ matrix $M$. We write $\textcolor{blue}{c_i(M,j)}$ to denote the $i$-th column of $M$ from row $j$ to row $n$, and \textcolor{blue}{$r_i(M,j)$} to denote the $i$-th row of $M$ from column $1$ to column $j.$ This proposition is obtained by a slight modification of the proof in \cite[Theorem 1.1.10]{Bru}.

\begin{proposition} \label{coxeterpermformula}
Let $M$ be an invertible real $n \times n$-matrix.
Then the column-permutation corresponding to $P[M]$ in the Bruhat decomposition of $M$ is given by
$$p_c(i)= \max \{ \ j \mid c_i(M,j) \text{ is not in the span of } \{ c_{1}(M,j), c_{2}(M,j), \cdots , c_{i-1}(M,j) \} \}.$$
Its inverse permutation, the row-permutation, can be computed as
$$p_r(i)= \min \{ \ j \mid r_i(M,j) \text{ is not in the span of } \{ r_{i+1}(M,j), r_{i+2}(M,j), \cdots , r_n(M,j) \} \}.$$
\end{proposition}

\begin{example}
\label{ex::Nakayama}
Let $A=KQ/I$ be the linear Nakayama algebra with quiver $Q:$
\[\begin{tikzcd}
	1 & 2 & 3 & 4 & 5
	\arrow["\alpha", from=1-1, to=1-2]
	\arrow["\beta", from=1-2, to=1-3]
	\arrow["\gamma", from=1-3, to=1-4]
	\arrow["\delta", from=1-4, to=1-5]
\end{tikzcd}\]
and relations $I=\langle \alpha \beta \gamma \delta \rangle.$
The Cartan matrix $\omega_A$ of $A$ is given by 
\[ \left( \begin{array}{ccccc}
1 & 0 & 0 & 0 & 0 \\
1 & 1 & 0 & 0 & 0 \\
1 & 1 & 1 & 0 & 0 \\
1 & 1 & 1 & 1 & 0 \\
0 & 1 & 1 & 1 & 1 \end{array} \right) \]
and the Coxeter matrix $C_A$ is given by
\[ \left( \begin{array}{ccccc}
0 & 0 & 0 & -1 & 0 \\
0 & 0 & 0 & 0 & -1 \\
-1 & 1 & 0 & 0 & -1 \\
-1 & 0 & 1 & 0 & -1 \\
-1 & 0 & 0 & 1 & -1 \end{array} \right) \]

A Bruhat factorisation of $C_A$ is given by 
\[ C_A=
\left( \begin{array}{ccccc}
-1 & 0 & 0 & 0 & 0 \\
0 & -1 & 0 & 0 & 0 \\
0 & 0 & 1 & 0 & -1 \\
0 & 0 & 0 & 1 & -1 \\
0 & 0 & 0 & 0 & -1
\end{array} \right)
\left( \begin{array}{ccccc}
0 & 0 & 0 & 1 & 0 \\
0 & 0 & 0 & 0 & 1 \\
0 & 1 & 0 & 0 & 0 \\
0 & 0 & 1 & 0 & 0 \\
1 & 0 & 0 & 0 & 0
\end{array} \right)
\left( \begin{array}{ccccc}
1 & 0 & 0 & -1 & 1 \\
0 & 1 & 0 & -1 & 0 \\
0 & 0 & 1 & -1 & 0 \\
0 & 0 & 0 & 1 & 0 \\
0 & 0 & 0 & 0 & 1
\end{array} \right)
\] 
We see now that the Coxeter permutation of $A$ is given by $p_A(1)=5,\  p_A(2)=3,\  p_A(3)=4,\  p_A(4)=1$ and $p_A(5)=2.$

\end{example}

Let $A$ be a Nakayama algebra. In \cite{R}, Ringel defined the following bijection, called \textcolor{blue}{\textit{the homological bijection}}:
$$\textcolor{blue}{h}: \{\text{ simple A-modules }\}/_{\cong} \rightarrow \{ \text{simple A-modules} \}/_{\cong}$$ $$S\mapsto h(S)=\top \Omega^{e(S)}N(S),$$
where $\textcolor{blue}{e(S)}=\min\{\pdim S,\pdim I(S)\},$ when $I(S)$ denotes the injective envelope of $S$ and $N(S)=I(S)$ if $e(S)$ is even and $N(S)=S$ if $e(S)$ is odd. Moreover as was shown by Ringel, we have that $e(S)=e^*(h(S))$, where $e^*$ is defined as $e^*(T)=\min\{\idim T,\idim P(T)\}$ for a simple $A$-module $T.$ 
Here $P(T)$ is the projective cover of $T.$
We define \textcolor{blue}{\emph{Ringel's homological permutation}} $\textcolor{blue}{\hat{h}}: \{1,...,n\} \rightarrow \{1,...,n\}$ as $\hat{h}(i)=j$ if $h(S_i)=S_j$.

\begin{lemma}\label{lem::e(S)odd_or_even}
    Let $A$ be a Nakayama algebra and $S$ a simple module. Then the following hold:
    \begin{enumerate}[\rm (a)]
        \item If $e(S)$ is odd, $e(S)=\pdim S\le \pdim I(S).$
        \item If $e(S)$ is even, $e(S)=\pdim I(S)\le \pdim S.$
    \end{enumerate}
\end{lemma} 
\begin{proof}
    This follows from a lemma by Madsen, see e.g. \cite[Lemma 3.2]{R}. 
\end{proof}

Another relevant bijection is due to Auslander and Reiten \cite[Proposition 5.4]{AR}.
Let $A$ be an Auslander-Gorenstein algebra.
Then there is a bijection, called the \textcolor{blue}{Auslander-Reiten bijection}
$$\textcolor{blue}{\psi}: \ \{\text{indecomposable injective A-modules }\}/_{\cong} \rightarrow \{\text{ indecomposable projective A-modules }\}/_{\cong}$$ $$I\mapsto P=\Omega^{d}(I),$$ where $d$ is the projective dimension of $I$.
Moreover, $P=\Omega^{d}(I)$ has injective dimension equal to $d$. 
The inverse of this bijection, $\psi^{-1}$ maps an indecomposable projective $A$-module $P$ to $I=\Omega^{-d}(P)$ where $d=\idim P.$ 
Since indecomposable injectives and indecomposable projectives are both in bijection with simple modules, this naturally gives rise to a permutation on the simple $A$-modules as well.
If $A$ has $n$ simple modules, we define the \textcolor{blue}{\emph{Auslander-Reiten permutation}} $\textcolor{blue}{\hat{\psi}}: \{1,\ldots,n\} \rightarrow \{1,\ldots,n\}$ as $\hat{\psi}(i)=j$ when $\psi(I(i))=P(j).$ 
We sometimes use the notation $\textcolor{blue}{\psi_A}$ and $\textcolor{blue}{\hat{\psi}_A}$ if we want to emphasise the algebra $A$ these maps correspond to.

It was shown in \cite{Kla} that one can characterise the Auslander-Gorenstein property of monomial algebras using the Auslander-Reiten bijection. This holds, in particular, for Nakayama algebras.

\begin{theorem}\cite[Theorem 7.9]{Kla}
\label{thm::AGcharacterisationviaARbij_monomial}
Let $A$ be a monomial algebra. Then $A$ is Auslander-Gorenstein if and only if the map $${\psi}: \ \{\text{indecomposable injective A-modules }\}/_{\cong} \rightarrow \{\text{ indecomposable projective A-modules }\}/_{\cong}$$ $$I\mapsto \Omega^{\pdim I}(I)$$ is well-defined and bijective.
Here, well-definedness means that $\pdim I<\infty$ and $\Omega^{\pdim I}(I)$ is indecomposable for every $I.$

\end{theorem}

It is an open question, due to Marczinzik \cite[Conjecture 1.2.4]{KM}, whether one can characterise the Auslander-Gorenstein property this way for every quiver algebra.

\begin{remark}
\label{rmk::inverse_ARbij}
\Cref{thm::AGcharacterisationviaARbij_monomial} also implies that a monomial algebra $A$ is Auslander-Gorenstein if and only if the map $\psi^{-1}: \{\text{indec. proj. $A$-modules}\}/_{\cong}\to \{\text{indec. inj. $A$-modules}\}/_{\cong}$ $P\mapsto \Omega^{-\idim P}(P)$ is well-defined and bijective. 
This holds since $A$ is Auslander-Gorenstein (resp. monomial) if and only if $A^{op}$ is, and $\psi_A^{-1}$ is well-defined and bijective if and only if the same holds for $\psi_{A^{op}}.$ 
\end{remark}

Iyama discovered a further permutation on the simple modules of an Auslander-Gorenstein algebra.
Namely, if $A$ is an Auslander-Gorenstein algebra, the \textcolor{blue}{\emph{grade bijection}} is defined as $$\textcolor{blue}{\phi}: \{\text{ simple A-modules }\}/_{\cong} \rightarrow \{ \text{ simple A-modules } \}/_{\cong}$$  $$S\mapsto \top(D\Ext_A^{g_S}(S,A)),$$ where $g_S:=\inf \{i \geq 0 \mid \Ext_A^i(S,A) \neq 0 \}$ is the grade of $S$. 
The \textcolor{blue}{\textit{cograde}} of the simple module $S$ is defined dually, i.e. as $\inf \{i \geq 0 \mid \Ext_A^i(D(A),S) \neq 0 \}.$ 
An important property of $\phi$ is that the cograde of $\phi(S)$ equals the grade of $S$, see  \cite[Theorem 2.10]{I}.
We define the \textcolor{blue}{\emph{grade permutation}} $\textcolor{blue}{\hat{\phi}}: \{1,\ldots,n\} \rightarrow \{1,\ldots,n\}$ as $\hat{\phi}(i)=j$ when $\phi(S_i)=S_j.$

An important theorem for Auslander-Gorenstein algebras is that Iyama's grade permutation and the Auslander-Reiten permutation coincide:

\begin{theorem}\cite[Theorem 3.12]{KMT}
\label{thm::gradebij=ARbij}
Let $A$ be an Auslander-Gorenstein algebra.
Then Iyama's grade permutation coincides with the Auslander-Reiten permutation.
\end{theorem}

\subsection{2-Gorenstein monomial algebras}
\label{subsec::monomial_results}

The paper \cite{Kla} investigates the Auslander-Gorenstein property of monomial algebras.
Here we recall the explicit characterisation of 2-Gorenstein monomial algebras in terms of quivers and relations and discuss some of their properties.
We use the following terminology: In a quiver algebra $KQ/I$ a path $p$ is a \textcolor{blue}{\textit{minimal relation}} if $p\in I$ and no proper subpath of $p$ is contained in $I.$

\begin{theorem}\cite[Theorem 4.1]{Kla}
\label{thm::2Gor_monomial} 
Let $A=KQ/I$ be a monomial algebra. Then $A$ is $2$-Gorenstein if and only if it satisfies the following conditions:
\begin{enumerate}[\rm (1)]
    \item $Q$ is biserial, i.e. for every vertex $x\in Q_0$ we have $\deg^{\text{out}}(x)\le 2$ and $\deg^{\text{in}}(x)\le 2.$
    \item For every vertex $x\in Q_0$ we have $\deg^{\text{out}}(x)=2$ if and only if $\deg^{\text{in}}(x)=2.$

    \item If $\deg(x)=4$ then we can label the incoming arrows at $x$ by $a_1$ and $a_2$ and the outgoing arrows from $x$ by $b_1$ and $b_2$, such that $a_1b_2,a_2b_1\in I,$ and $a_ib_i$ are not contained in any minimal relations for $i=1,2$.
{

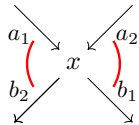
\centering
\resizebox{0.15\linewidth}{!}{
\begin{tikzpicture}
    \node (A) at (0, 0) {\large $x$};
    \node (B) at (1, 1) {};
    \node (C) at (-1, 1) {};
    \node (D) at (1, -1) {};
    \node (E) at (-1, -1) {};

     \node (v) at (0.5,0.5) {};
    \node (w) at (0.5,-0.5) {};
    \node (x) at (-0.5,0.5) {};
    \node (y) at (-0.5,-0.5) {};


    \draw[red, bend left, shorten <=1pt, shorten >=1pt, line width=1pt] (v) to (w);
    \draw[red, bend right, shorten <=1pt, shorten >=1pt, line width=1pt] (x) to (y);

    \draw[->] (B) to (A) ;
    \draw[->] (C) to (A) node[midway,xshift=-0.8cm,yshift=0.4cm ] {$a_1$};
    \draw[->] (A) to (D) node[midway,xshift=-0.8cm,yshift=-0.4cm ] {$b_2$};
    \draw[->] (A) to (E) node[midway,xshift=0.8cm,yshift=-0.4cm ] {$b_1$};
    \draw[->] (A) to (E) node[midway,xshift=0.8cm,yshift=0.4cm ] {$a_2$}; 
\end{tikzpicture}
}

}

\item If an arrow is contained in a minimal relation, then this arrow must be the start or the end of some minimal relation. 

\end{enumerate}

In particular, every $2$-Gorenstein monomial algebra is a string algebra.
\end{theorem}

This result links $2$-Gorenstein monomial algebras to string algebras, which makes them much more accessible to work with.
String algebras are a well-researched class of algebras, we understand how their indecomposable modules look and how irreducible morphisms between such modules are described. 
This allows for explicit calculations of projective and injective resolutions, which we make use of in the following lemma. 
For a reference on string algebras, see \cite{BR} and \cite{CB}. 

For the proof of the next lemma, note that in a monomial algebra $A=KQ/I$ there exists a canonical uniserial $A$-module corresponding to any non-zero path $p\notin I.$ We denote this by $\textcolor{blue}{M(p)}$ in the following.

\begin{lemma}
\label{lem:InjResOfProj_2gorensteinmonomial}
    Let $A=KQ/I$ be a $2$-Gorenstein monomial algebra, $i\in Q_0$, and consider a minimal injective coresolution of the indecomposable projective module $P(i):$ 
    $$0\to P(i) \to I^0 \to I^1 \to I^2 \to \ldots. $$
     Then the following hold:
    \begin{enumerate}[(a)]
        \item If $P(i)$ is uniserial, then $I^d$ is indecomposable for every $\idim P(i)\ge d\ge 0.$
        \item If $P(i)$ is not uniserial, then a minimal injective coresolution is given by
        \begin{equation}
            \label{eq:nonuniserial_proj}
            0\to P(i) \to I(t_1)\oplus I(t_2) \to I(i) \to 0,
        \end{equation}
    where $t_1$ and $t_2$ stand for the end vertices of the two maximal non-zero paths in $A$ starting at $i.$
    \end{enumerate}

\end{lemma}
\begin{proof}
Although these statements follow from the proofs of \cite[Lemma 4.13]{Kla} and \cite[Lemma 5.3]{Kla}, for the convenience of the reader, we include a direct proof here.
Assume first that $P(i)$ is uniserial. 
We will prove the following claim by induction on $j.$

\begin{claim*}
     $\Omega^{-j}(P(i))$ is uniserial for every $\idim P(i)\ge j\ge 0$.
\end{claim*}
This implies that every non-zero $I^j$ has a simple socle and is therefore indecomposable.

\begin{proof}[Proof of the Claim]
    Since $P(i)$ is uniserial, the statement holds for $j=0$, and there exists a unique maximal non-zero path $p$ starting at $i.$
    We have $I^0 \cong I(t(p))$ where $t(p)$ denotes the end vertex of $p.$
    By the maximality of $p$ and \Cref{thm::2Gor_monomial} (3), $t(p)$ has at most one incoming arrow. 
    Since $A$ is a string algebra by \Cref{thm::2Gor_monomial}, this ensures that $I^0$ is uniserial.
    Thus, all factor modules of $I^0$ are uniserial.
    This proves the Claim for $j=1.$

    Now assume that the Claim holds for a $1\le j < \idim P(i),$ and prove that this implies the statement for $j+1$ as well. 
    Let the uniserial $\Omega^{-j}(P(i))$ correspond to the non-zero path $q$ in $Q,$ i.e. $\Omega^{-j}(P(i))\cong M(q).$ 
    We have $I^j\cong I(t(q)).$
    Now if $t(q)$ has at most one incoming arrow, then since $A$ is a string algebra, both $I^j$ and its non-zero factor module $\Omega^{-(j+1)}(P(i))$ are uniserial.
    So it remains to consider the case when $t(q)$ has two incoming arrows. 
    By \Cref{thm::2Gor_monomial} (2), this means that $t(q)$ has degree $4.$ 
    As $A$ is a string algebra, this means that there are exactly two maximal non-zero paths ending at $t(q)$, $i_1$ and $i_2,$ and these end with different arrows. 
    One of them ends with $q,$ assume that this is $i_1.$
    (A priori, it is possible that $q=e_{t(q)}$, in which case both $i_1$ and $i_2$ and with $q.$)
    As a next step, we prove that $i_1=q.$ 
    The module $I^j\cong I(t(q))$ can be pictured as follows:

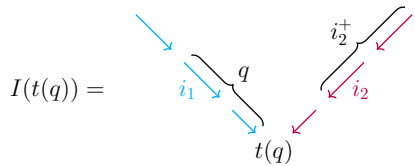
\begin{figure}[h]
    \resizebox{.35\linewidth}{!}{\input{inj.tex}}
\end{figure}

    Since $j\ge 1, $ $\Omega^{-j}(P(i))\cong M(q)$ is a factor module of an injective module. 
    Since it is uniserial by assumption, it is a factor module of an indecomposable injective $I(y).$
    Thus, there exists a path $r$ such that $qr$ is a maximal non-zero path ending at $y.$
    Thus, $aqr\in I$ for every $a\in Q_1.$
    First of all, this ensures that $q\neq e_{t(q)}.$
    Indeed, \Cref{thm::2Gor_monomial} (3) implies that if there is a non-zero path starting at a degree 4 vertex, in this case  $qr=e_{t(q)}r=r\notin I$, then there is always an arrow $a$ such that $aqr\notin I.$ This, however, contradicts the maximality of $qr,$ hence, $q$ contains at least one arrow. 
    By \Cref{thm::2Gor_monomial} (3), a minimal relation in $A$ which consists of at least three arrows can only contain a degree 4 vertex as one of its endpoints.
    So since $t(q)$ has degree 4,
    $aqr\in I$ implies that we must already have $aq\in I.$ This is true for all $a\in Q_1.$ So $q$ is a maximal path ending at $t(q),$ which forces $i_1=q.$

    Hence, $\Omega^{-(j+1)}(P(i))\cong \coker (\Omega^{-j}(P(i)) \hookrightarrow I^j)\cong I(t(q))/M(i_1)\cong  M(i_2^+)$, where $i_2^+$ is the path obtained from $i_2$ by omitting its last arrow. 
    This proves the Claim for $j+1$.    
\end{proof}

It remains to prove (b). 
For this, assume that $P(i)$ is not uniserial.
Since $A$ is a string algebra, there are exactly two non-zero paths starting at $i,$ $p_1$ and $p_2,$ ending with vertices $t_1$ and $t_2$, respectively.
We have $I^0\cong I(t_1)\oplus I(t_2).$
The same argument as in case (a) tells us that $I(t_1)$ and $I(t_2)$ are both uniserial, corresponding to the non-zero paths $i_1p_1$ and $i_2p_2,$ respectively.
In particular, $ai_jp_j\in I$ for all $a\in Q_1$ and $j=1,2.$
By \Cref{thm::2Gor_monomial}, $i$ is a degree $4$ vertex.
Part (3) of \Cref{thm::2Gor_monomial} tells us that if we have a non-zero path $q$ starting at a degree 4 vertex, then there is always an arrow $b$ such that $bq\notin I.$ 
Thus, both $i_1$ and $i_2$ consist of at least one arrow.
As mentioned before, a minimal relation in $A$ consisting of at least three arrows can only contain a degree 4 vertex as one of its endpoints.
Thus, $ai_jp_j\in I$ forces $ai_j\in I$ for all $a\in Q_1$ and $j=1,2.$ 
This implies that $i_1$ and $i_2$ are the two  maximal non-zero paths ending at $i.$
Finally, \Cref{thm::2Gor_monomial} (3) ensures that $i_1$ and $i_2$ end with different arrows, and $p_1$ and $p_2$ start with different arrows.
Then calculation gives that $\coker (P(i)\hookrightarrow I(t_1)\oplus I(t_2))\cong I(i).$
The resolution (\ref{eq:nonuniserial_proj}) can be pictured as:

\begin{figure}[h]
    \resizebox{.85\linewidth}{!}{\input{inj_resol.tex}}
\end{figure}
\end{proof}

\begin{corollary}
\label{cor:decreasing_labels_injective_resol}
    Let $A=KQ/I$ be a $2$-Gorenstein monomial algebra, and let $i,x,y \in Q_0.$ 
    If $$0\to P(i)\to I^0 \to I^1 \to \ldots $$ is a minimal injective coresolution of $P(i)$ and $I(x)$ is a direct summand of $I^j$ and $I(y)$ is a direct summand of $I^k$ for some $0\le j<k\le \idim P(i)$ then there exists a path $p$ starting at $y$ and ending at $x$ in $Q.$ (Here $p\in I$ is possible.)
\end{corollary}
\begin{proof}
    If $P(i)$ is non-uniserial, the statement follows immediately from Lemma \ref{lem:InjResOfProj_2gorensteinmonomial} (b). 
    Thus, assume that $P(i)$ is uniserial. 
    Then part (a) of the same lemma tells us that for every $0\le j\le \idim P(i)$ we have $I^j\cong I(x_j)\cong D(e_{x_j}A)$ for some vertex $x_j.$ 
    Thus, all composition factors of $I^j$ are isomorphic to simple modules $S_v$ where $v\in Q_0$ so that there is a (non-zero) path from $v$ to $x_j.$
    This is in particular true for the (indecomposable) socle of $I^{j+1}$, so there is always a path from $x_{j+1}$ to $x_{j}.$
    Concatenating these paths, we obtain a path $p$ from $x_k$ to $x_j$ for all $0\le j<k\le \idim P(i)$ (which may be contained in $I$ if $j+1<k.$)

\end{proof}

\subsection{Nakayama algebras}
\label{subsec::Nakayama_notation}
Finally, let us recall some preliminaries and notation on Nakayama algebras. 
Let $A=KQ/I$ be a connected Nakayama algebra with $n$ simple modules and 
underlying quiver $Q$ that is either a linearly oriented line as in (\ref{eq:linear_quiver}) or a linearly oriented cycle as in (\ref{eq:cyclic_quiver}). 
Then $A$ is uniquely determined by its \textcolor{blue}{\textit{Kupisch series}} $(c_1,\ldots, c_n)$ where $c_i=\dim P(i)=\dim e_iA\in \mathbb{Z}^+.$ 
Dually, $A$ is also uniquely determined by its \textcolor{blue}{\textit{co-Kupisch series}} $(d_1,\ldots, d_n)$ where $d_i=\dim I(i)=\dim D(Ae_i)\in \mathbb{Z}^+.$ 
To deal with the cyclic case, it will be convenient for us to regard the entire $\mathbb{Z}$ as the vertex set of
$Q$. 
For this, we set $e_i:=e_{\overline{i}}$, $c_i:=c_{\overline{i}}$,  $d_i:=d_{\overline{i}},$ $S_i := S_{\overline{i}}$, $I(i) := I(\overline{i})$ and $P(i) := P(\overline{i})$ for every $i\in \mathbb{Z}$  where $\overline{i}\in Q_0 = \{1,\ldots, n\}$
such that $\overline{i} \equiv i \ (\mod n )$. 
Every indecomposable $A$-module is uniserial and is of the form $e_iA/e_iJ^k$ for some $i\in \mathbb{Z}$ and $0\le k\le c_i.$ 
Here $J$ denotes the Jacobson radical of $A.$
The unique composition series of $e_iA/e_iJ^k$ is $0\subsetneq e_iJ^{k-1}/e_iJ^k \subsetneq e_iJ^{k-2}/e_iJ^k \subsetneq \ldots \subsetneq e_iA/e_iJ^k$, and the corresponding composition factors are $S_{i+k-1}, S_{i+k-2}, \ldots, S_{i+1},S_i. $
So we can naturally associate the indecomposable module $e_iA/e_iJ^k$ with the interval $[i,i+k-1]\subsetneq \mathbb{Z}.$
The intervals $[i,j]$ and $[i+cn,j+cn]$ denote the same modules for every $c\in \mathbb{Z}.$
The Loewy length of $e_iA/e_iJ^k$ is equal to the length of the interval $[i,i+k-1]$, that is $k.$
Note that $\top e_iA/e_iJ^k\cong S_i$ and $\soc e_iA/e_iJ^k\cong S_{i+k-1}.$ 
We will often abuse notation and associate a vertex $i$ with the corresponding simple $S_i.$ 
We write $S_i\le S_j$ to mean $i\le j$ in $\mathbb{Z}.$ For this, it is important to work with the specific choices of $i,j\in \mathbb{Z},$ not modulo $n.$ 

Let us define the following maps for $A$:
$$\textcolor{blue}{f}: \mathbb{Z}\rightarrow \mathbb{Z}, \ j \mapsto j+c_j$$ and 
$$\textcolor{blue}{g}: \mathbb{Z}\rightarrow \mathbb{Z}, \ j \mapsto j-d_j.$$
These maps $f$ and $g$ were already introduced in \cite{Gu}, and their properties were studied, e.g., in \cite{Sh} and \cite{R}.
It is important to note that $f$ and $g$ are both monotone increasing, and can be used to describe homological resolutions of indecomposable modules. 
By definition, $P(j)=[j,f(j)-1]$ and $I(j)=[g(j)+1,j].$  
Then a minimal injective coresolution of an indecomposable module $M$ is given by:
\begin{equation}
\label{eq::injResolUsingg}
    0\rightarrow M \rightarrow {I(a)} \rightarrow {I(b)} \rightarrow {I(g(a))} \rightarrow {I(g(b))} \rightarrow {I(g^2(a))} \rightarrow {I(g^2(b))}\rightarrow \ldots
\end{equation}
where $M=[b+1,a] \subsetneq \mathbb{Z}.$ Dually, a minimal projective resolution is given by: 
\begin{equation}
\label{eq::projResolUsingf}
    \ldots\rightarrow {P(f^2(d))} \rightarrow {P(f^2(c))} \rightarrow P(f(d)) \rightarrow {P(f(c))} \rightarrow {P(d)} \rightarrow {P(c)} \rightarrow  {M} \rightarrow 0
\end{equation}
where $M=[c,d-1] \subsetneq \mathbb{Z}.$
We have
\begin{equation}
    \idim M= \inf \ (  \{2i+1\in \mathbb{N}_0\ | \ g^{i+1}(a)=g^{i+1}(b)\} \cup \{2i\in \mathbb{N}_0\ | \ g^{i+1}(a)=g^{i}(b)\}  ) 
\end{equation}
and
\begin{equation}
   \pdim M= \inf \ ( \{2i+1\in \mathbb{N}_0\ | \ f^{i+1}(c)=f^{i+1}(d)\} \cup \{2i\in \mathbb{N}_0\ | \ f^{i+1}(c)=f^{i}(d)\}  ). 
\end{equation}

The following lemma plays an important role in the proofs of the next chapter. 
\begin{lemma}
    \label{lem::inequalitiesSocTop}
    Let $A$ be a Nakayama algebra, and let $M$, $I,$ and $P$ be indecomposable $A$-modules, such that $P$ is projective and $I$ is injective. Then the following hold:

    \begin{enumerate}[\rm (a)]
        \item If $\soc (P) <  \soc(M)$ then $\top (P) < \top(M)$.
        \item If $\soc (I) \le  \soc(M)$ then $\top (I) \le \top(M)$. 
        
    \end{enumerate}
    
\end{lemma}
\begin{proof}
    The proof is spelled out in \cite[Lemma 6.4]{Kla}.
\end{proof}

\begin{corollary}
\label{cor::Nakayama_f_g_formulas}
Let $A$ be a Nakayama algebra. Let $j,\ell\in \mathbb{Z}.$ Then the following hold:
\begin{enumerate}[\rm (a)]
    \item If $f(\ell)-1<j$ then $\ell-1 < g(j).$ 
    \item If $j\le f(\ell)-1$ then $g(j) \le \ell-1.$
\end{enumerate}

\end{corollary}
\begin{proof}
    Apply Lemma \ref{lem::inequalitiesSocTop} to the projective $P(\ell)=[\ell,f(\ell)-1]$ and the injective $I(j)=[g(j)+1,j].$ This is also part of \cite[Corollary 6.5]{Kla}.
    
\end{proof}

\section{Ringel's bijection and the Coxeter permutation}
We show that Ringel's homological permutation coincides with the Coxeter permutation for linear Nakayama algebras. Note that the Coxeter permutation depends on the ordering of simple modules. For linear Nakayama algebras, there exists a canonical ordering, we are going to fix this for this chapter:
\begin{equation}
\label{eq:canonical_labelling}
    1 \longrightarrow 2 \longrightarrow \ldots \longrightarrow n.
\end{equation}
Thus, when we talk about the Coxeter permutation of a linear Nakayama algebra, we always calculate it with respect to this canonical ordering.
As the second main result of this chapter (\Cref{thm::del(S)=e(S)}), we prove that in every Nakayama algebra, the delooping level of a simple module $S$ coincides with $e(S)$, the invariant appearing in Ringel's homological bijection.

\begin{theorem}
\label{thm::h=Coxeter_perm}
    Let $A$ be a linear Nakayama algebra. Then Ringel's homological permutation for $A$ coincides with the Coxeter permutation of $A.$
\end{theorem}

For the proof of Theorem \ref{thm::h=Coxeter_perm} (as well as for Theorem \ref{thm::del(S)=e(S)}), we will need the following (quite technical) lemmas.

We refer the reader to \Cref{ex:M}, which demonstrates the statement of these lemmas.

\begin{lemma}
\label{lem::existanceOfM}
    Let $A=KQ/I$ be a Nakayama algebra with a vertex $i$ such that $e(S_i)\ge 2.$ Then there exists an indecomposable $A$-module $M$ with a minimal injective coresolution \begin{equation*}
    0\longrightarrow M \longrightarrow I^0  \longrightarrow I^1 \longrightarrow \ldots \longrightarrow I^d \longrightarrow 0
\end{equation*} such that the following statements hold:
    \begin{enumerate}[\rm (a)]
        \item $I^d\cong I(i)$
        \item $d=\idim M= e(S_i)-1.$
    \end{enumerate}
\end{lemma}
\begin{proof}
We start by fixing a minimal projective resolution of $S_i$ and one of $I(i):$

\[
\ldots 
\longrightarrow \textcolor{orange}{P_3} 
\longrightarrow \textcolor{purple}{P_2} 
\longrightarrow \textcolor{orange}{P_1} 
\longrightarrow \textcolor{purple}{P_0} 
\longrightarrow S_i 
\longrightarrow 0
\]

and  
\[
\ldots 
\xrightarrow{d_4} \textcolor{orange}{R_3} 
\xrightarrow{d_3} \textcolor{blue}{R_2} 
\xrightarrow{d_2} \textcolor{orange}{R_1} 
\xrightarrow{d_1} \textcolor{blue}{R_0} 
\xrightarrow{d_0} I(i) 
\longrightarrow 0.
\]

It is quite natural to look at these resolutions as they appear in the definition of $e(S_i)$ and $h(S_i).$
In the following, we will prove (a) and (b) for a specific choice of $M$. Consider the natural inclusions $\ker (d_{e(S_i)-1}) \hookrightarrow R_{e(S_i)-1} \hookrightarrow I(R_{e(S_i)-1}),$ where $I({R_{e(S_i)-1}})$ denotes the injective envelope of ${R_{e(S_i)-1}}.$ 
Via this, we can regard $\ker (d_{e(S_i)-1})$ as a submodule of $I(R_{e(S_i)-1})$ and set $M$ to be the factor module 
$$M\cong I(R_{e(S_i)-1})/\ker (d_{e(S_i)-1}).$$

The idea behind this particular choice of $M$ is that in this case, we can approximate the terms in a minimal injective coresolution of $M$ using the projective resolutions of $S_i$ and $I(i).$ 
We dedicate the rest of this proof to showing that this is a correct choice. This will be quite technical. 

We will use the notation introduced in Section \ref{subsec::Nakayama_notation}, and think of all terms in these resolutions as intervals in $\mathbb{Z}.$
Then each resolution consists of two sequences of neighbouring intervals, one for the odd terms, and one for the even terms.
This is illustrated in \Cref{fig:resolutions}.
Since we will compare socles and tops of two modules as integers, it will be important not to work modulo $n$.
So let us fix vertex $i$ as a specific element of $\mathbb{Z}.$ 
Then $S(i)=[i]$, $I(i)=[g(i)+1,i]$ and using (\ref{eq::projResolUsingf}), one can express the even terms in the projective resolution of $S_i$ (resp. $I(i)$) as $\textcolor{purple}{P_{2r}}=P(f^r(i))=[f^r(i),f^{r+1}(i)-1]$ (resp. $\textcolor{blue}{R_{2r}}=P(f^r(c))=[f^r(c),f^{r+1}(c)-1]$ where $c=\top I(i)$) and the odd terms as $\textcolor{orange}{P_{2r+1}}=P(f^r(i+1))=[f^r(i+1),f^{r+1}(i+1)-1]=\textcolor{orange}{R_{2r+1}}.$
In particular,
\begin{equation}
\label{eq::odd_terms_coincide}
     \textcolor{orange}{P_{2r+1}}= \textcolor{orange}{R_{2r+1}}
\end{equation}
for every $e(S_i)\ge 2r+1\ge 0$.
Importantly, this is not only an isomorphism of modules, but an equality of the corresponding intervals in $\mathbb{Z}.$

Note that since $2\le e(S_i)\le \pdim I(i)$ and since $I(R_{e(S_i)-1})$ is uniserial, we have 
\begin{equation}
\label{eq:socM}
    \soc M \cong \soc \im(d_{e(S_i)-1}) \cong \soc R_{e(S_i)-2}.
\end{equation}
Then we associate $M$ with the interval $[b+1,j]\subsetneq \mathbb{Z}$, where $j=\soc R_{e(S_i)-2}.$  
Using (\ref{eq::injResolUsingg}), we can express the even terms in the resolution of $M$ as $\textcolor{ForestGreen}{I^{2r}}=I(g^r(j))=[g^{r+1}(j)+1,g^r(j)]$ and the odd terms as $\textcolor{ForestGreen}{I^{2r+1}}=I(g^r(b))=[g^{r+1}(b)+1,g^r(b)]$.

To prove that $M$ satisfies (a) and (b), we distinguish two cases.
First, we consider the case when $e(S_i)$ is odd. By Lemma \ref{lem::e(S)odd_or_even}, we have $\pdim S_i=e(S_i)$ in this case.  

We prove $\idim M \ge e(S_i)-1$ by showing 
$\top \textcolor{ForestGreen}{I^{r}}<\top \textcolor{ForestGreen}{I^{r-1}}$ for all $r<e(S_i)-1$. We do this by approximating the tops of the $\textcolor{ForestGreen}{I^r}$ by the tops of the terms appearing in the projective resolutions of the modules $S_i$ and $I(i).$ 
More precisely, we will show that  

\begin{equation}
\label{eq::thirdApproximation}
    \top \textcolor{orange}{R_{e(S_i)-2(r+1)}}< \top \textcolor{ForestGreen}{I^{2r-1}} \le \top \textcolor{blue}{R_{e(S_i)-(2r+1)}}
\end{equation}
\begin{equation}
\label{eq::fourthApproximation}
 \top \textcolor{purple}{P_{e(S_i)-(2r+1)}}< \top \textcolor{ForestGreen}{I^{2(r-1)}}\le
    \top \textcolor{orange}{P_{e(S_i)-2r}}
\end{equation}
\begin{equation}
\label{eq::fifthApproximation}
    \top \textcolor{blue}{R_{e(S_i)-(2r+1)}}\le \top \textcolor{purple}{P_{e(S_i)-(2r+1)}}.
\end{equation}

Here, the first inequality in (\ref{eq::thirdApproximation}) holds for every $1\le r\le \frac{e(S_i)-3}{2}$, while all other inequalities work for all $1\le r\le \frac{e(S_i)-1}{2}$.
Then (\ref{eq::thirdApproximation}), (\ref{eq::fourthApproximation}) and (\ref{eq::fifthApproximation}) combined prove $\top \textcolor{ForestGreen}{I^r}< \top \textcolor{ForestGreen}{I^{r-1}}$ for all $1\le r< e(S_i)-1.$ The inequalities (\ref{eq::thirdApproximation}) and (\ref{eq::fourthApproximation}) are depicted in \Cref{fig:resolutions}.

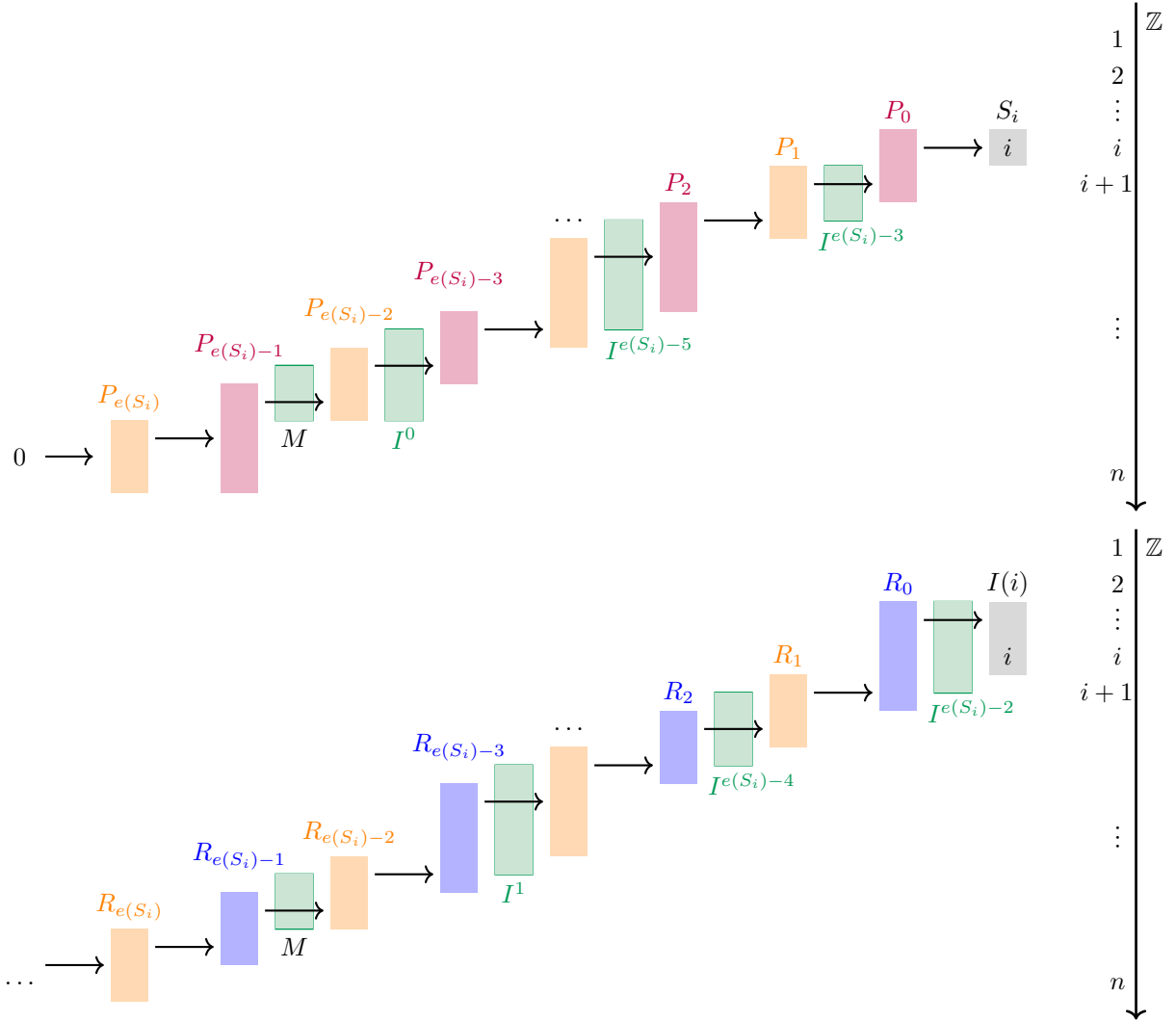
\begin{figure}
    \centering
    \input{fig_resolutions}
    \caption{A picture of a minimal projective resolution of $S_i$ (top) and $I(i)$ (bottom). The modules are depicted as intervals in $\mathbb{Z}$. In the resolution of $S_i$ (resp. $I(i)$), the even terms are shown in purple (resp. blue), while the odd terms are orange. Note that $\textcolor{orange}{P_j=R_j}$ for all odd $j.$ This picture corresponds to the case when $e(S_i)$ is odd.
    The green intervals represent the terms of a minimal injective coresolution of $M$, the module defined in the proof of Lemma \ref{lem::existanceOfM}.
    The diagram indicates how these intervals align with the terms of the two resolutions. 
    }
    \label{fig:resolutions}
\end{figure}

We can rephrase the above inequalities as

\begin{equation}
\label{eq::thirdApproximation_rephrased}
    f^{\frac{e(S_i)-2r-3}{2}}(i+1)< g^r(b)+1 \le f^{\frac{e(S_i)-2r-1}{2}}(c)
\end{equation}
\begin{equation}
\label{eq::fourthApproximation_rephrased}
  f^{\frac{e(S_i)-2r-1}{2}}(i)< g^{r}(j)+1 \le f^{\frac{e(S_i)-2r-1}{2}}(i+1)
\end{equation}
\begin{equation}
\label{eq::fifthApproximation_rephrased}
   f^{\frac{e(S_i)-2r-1}{2}}(c)\le f^{\frac{e(S_i)-2r-1}{2}}(i).
\end{equation}

We will prove these by induction. Note that if (\ref{eq::thirdApproximation_rephrased}) holds for some $r$, then we can apply Corollary \ref{cor::Nakayama_f_g_formulas} to obtain that it must also hold for $r+1$ provided the exponent of $f$ is still non-negative. For the first inequality, this is the condition $\frac{e(S_i)-2r-3}{2}\ge 1$, which is equivalent to $r+1\le \frac{e(S_i)-3}{2}.$ In the second, the condition is $r+1\le \frac{e(S_i)-1}{2}.$
The same is true for (\ref{eq::fourthApproximation_rephrased}). 
For (\ref{eq::fifthApproximation_rephrased}), the induction works in the other direction. 
Namely, the monotonicity of $f$ implies that if the inequality holds for an $r,$ it also holds for $r-1.$
So in all of these cases, it remains to show the base case of the induction.

First of all, $f^0(c)=c=\top I(i)\le \soc I(i) = i=f^0(i).$ 
This proves the base case for (\ref{eq::fifthApproximation_rephrased}).

Second, note that 
$\soc \textcolor{purple}{P_{e(S_i)-3}}<\soc M =\soc \textcolor{orange}{P_{e(S_i)-2}}.$ 
This holds as $\pdim S_i=e(S_i)\ge 3.$
This inequality can be rephrased as $f^{\frac{e(S_i)-1}{2}}(i)-1<j=f^{\frac{e(S_i)-1}{2}}(i+1)-1.$
Applying Corollary \ref{cor::Nakayama_f_g_formulas} yields exactly the base case for (\ref{eq::fourthApproximation_rephrased}).

Finally, observe that  $\soc\textcolor{orange}{R_{e(S_i)-2}}<\soc I(\textcolor{blue}{R_{e(S_i)-1}})= \soc\textcolor{blue}{R_{e(S_i)-1}}.$ 
The first inequality holds since $\pdim I(i)\ge e(S_i).$
Then we can apply Lemma \ref{lem::inequalitiesSocTop} to obtain $\top\textcolor{orange}{R_{e(S_i)-2}}<\top I(\textcolor{blue}{R_{e(S_i)-1}})\le  \top\textcolor{blue}{R_{e(S_i)-1}}$.
Since $b+1=\top M = \top I(\textcolor{blue}{R_{e(S_i)-1}})$, this is equivalent to $f^{\frac{e(S_i)-3}{2}}(i+1)<b+1\le f^{\frac{e(S_i)-1}{2}}(c).$ 
Applying Corollary \ref{cor::Nakayama_f_g_formulas} yields the base case for the induction for (\ref{eq::thirdApproximation_rephrased}).
This concludes the proof of $$\idim M\ge e(S_i)-1.$$

Note that (\ref{eq::fourthApproximation}) for $r=\frac{e(S_i)-1}{2}$ yields $i=\top \textcolor{purple}{P_0}<\top \textcolor{ForestGreen}{I^{e(S_i)-3}}\le \top \textcolor{orange}{P_1}=i+1.$ Hence, $\top \textcolor{ForestGreen}{I^{e(S_i)-3}}=i+1$ and $\soc \textcolor{ForestGreen}{I^{e(S_i)-1}}=i,$ which proves 
$$\textcolor{ForestGreen}{I^{e(S_i)-1}}\cong I(i).$$ 

The inequality (\ref{eq::thirdApproximation}) for $r=\frac{e(S_i)-1}{2}$ yields $\top \textcolor{ForestGreen}{I^{e(S_i)-2}}\le \top \textcolor{blue}{R_0} = \top I(i).$ Thus, thanks to the isomorphism above, $\top \textcolor{ForestGreen}{I^{e(S_i)-2}}\le \top \textcolor{ForestGreen}{I^{e(S_i)-1}},$ which further implies $\top \textcolor{ForestGreen}{I^{e(S_i)-2}}= \top \textcolor{ForestGreen}{I^{e(S_i)-1}}$ and $$\idim M=e(S_i)-1.$$

This concludes the proof of the lemma in the odd case. The case when $e(S_i)$ is even can be done analogously. In that case, the initial inequalities in the base case are  $\soc \textcolor{orange}{R_{e(S_i)-3}}< \soc \textcolor{ForestGreen}{I^0} = \soc \textcolor{blue}{R_{e(S_i)-2}}$ and  $\soc \textcolor{purple}{P_{e(S_i)-4}}< \soc \textcolor{ForestGreen}{I^1} \le \soc \textcolor{orange}{P_{e(S_i)-3}}$. These allow us to prove inequalities analogous to (\ref{eq::thirdApproximation}) and (\ref{eq::fourthApproximation}).
\end{proof}

For the next lemma, recall that if $B$ is an $n\times n$ matrix and $i,j\in \{1,\ldots,n\}$ then $r_j(B,i)$ denotes the $j$-th row of $B$ from column $1$ to column $i$ as defined for \Cref{coxeterpermformula}. 
If we have a $1\times n$ vector $v=(v_1,\ldots,v_n)$ then set $\textcolor{blue}{v_{[i:j]}}=(v_i,v_{i+1},\ldots,v_j).$ 
We also remind the reader that $\hat{h}$ stands for Ringel's homological permutation as defined in \Cref{subsec::permutations_on_simples}.

\begin{lemma}
\label{lem::dimvectorofM_linear}
    Let $A=KQ/I$ be a linear Nakayama algebra with vertices $Q_0=\{1,2,\ldots,n\}$. Let $i\in Q_0$ such that $e(S_i)\ge 2.$ Then in Lemma \ref{lem::existanceOfM}, the module $M$ can be chosen such that 
    
    $$\underline{\dim}(M)_{[1:\hat{h}(i)-1]}=r_j(\omega_A,\hat{h}(i)-1) \text{ and }
    \underline{\dim}(M)_{[\hat{h}(i):n]}=0$$
    for some $j>i.$ In particular, $M$ contains no simple $S_{\ell}$ with $\ell\ge \hat{h}(i)$ as a composition factor.
\end{lemma}

\begin{proof}
In the proof of Lemma \ref{lem::existanceOfM}, we have shown that $M$ can be chosen to be a factor module of $I(j)$, where $S_j\cong \soc R_{e(S_i)-1}$. 
As $R_{e(S_i)-1}$ is part of the minimal projective resolution of $I(i)$ and $I(i)$ is not projective, we have $j>i.$

By (\ref{eq:socM}), this $M$ has socle isomorphic to $\soc R_{e(S_i)-2}\cong S_m.$ 
Since $A$ is a linear Nakayama algebra, all composition factors of $M$ are different, and $M$ has no composition factor $S_{\ell}$ with $\ell>m.$ 
Our goal is to show that $m=\hat{h}(i)-1,$ which is equivalent to $\top R_{e(S_i)}\cong S_{h(S_i).}$
If $e(S_i)$ is even, $\top R_{e(S_i)}\cong S_{h(S_i)}$ by the definition of $h$.
If $e(S_i)$ is odd, $\top P_{e(S_i)}\cong S_{h(S_i)},$ but recall that in this case $R_{e(S_i)}= P_{e(S_i)}$ by (\ref{eq::odd_terms_coincide}), so the statement still holds. 
In conclusion, $m=\hat{h}(i)-1.$

Finally, recall that the $\ell$-th row in the Cartan matrix of $A$ is the dimension vector of $I(\ell).$ 
This, combined with our previous observations, implies 
$\underline{\dim}(M)_{[1:\hat{h}(i)-1]}=r_j(\omega_A,\hat{h}(i)-1)$ and $\underline{\dim}(M)_{[\hat{h}(i):n]}=0$.
\end{proof}

\begin{remark}
    \label{rmk::cyclic_case}
    In the cyclic case, the statement of Lemma \ref{lem::dimvectorofM_linear} is false in general for the choice of $M$ which was described in the proof of Lemma \ref{lem::existanceOfM}.    
     
\end{remark}

Let us demonstrate the statements of \Cref{lem::existanceOfM} and \ref{lem::dimvectorofM_linear} on an example.

\begin{example}
\label{ex:M}
    Let $A$ be the linear Nakayama algebra from \Cref{ex::Nakayama}. 
    We have $S_1=I(1)$ and calculation gives $e(S_1)=\pdim(S_1)=\pdim (I(1))=2.$
    Indeed, we have the following minimal projective resolution:
    $$0\rightarrow P(5) \xrightarrow{d_2} P(2) \xrightarrow{d_1} P(1) \rightarrow S_1 = I(1) \rightarrow 0, $$
    where we can describe the terms appearing in terms of their composition series as 
    $$P(5)=5, \ P(2)= \begin{smallmatrix}
        2 \\
        3 \\
        4 \\
        5 \\
    \end{smallmatrix}, \
    P(1)= \begin{smallmatrix}
        1\\
        2 \\
        3 \\
        4 \\
    \end{smallmatrix},\
    S_1=I(1)= 1. 
    $$
    Then following the proof of \Cref{lem::existanceOfM}, set $M:= I(P_1)/\ker(d_1)\cong P(2)/P(5) \cong \begin{smallmatrix}
        2 \\
        3 \\
        4 \\
    \end{smallmatrix}$. 
    This module has a minimal injective coresolution
    $$0\rightarrow M \rightarrow I(4) \rightarrow I(1) \rightarrow 0.$$
    Thus, we have $\idim(M)=1=e(S_1)-1$ and the last term in the above resolution is $I(1),$ as predicted by \Cref{lem::existanceOfM}.
    
    Moreover, we obtain $\underline{\dim}(M)=(0,1,1,1,0)$ and $\hat{h}(1)=5$.
    Then the proof of \Cref{lem::dimvectorofM_linear} tells us that we should consider row $5$ of the Cartan matrix of $A.$ This was already calculated in \Cref{ex::Nakayama} to be $(0,1,1,1,1)$.
    This yields $\underline{\dim}(M)_{[1:4]}=(0,1,1,1)=r_5(\omega_A,4)$ and $\underline{\dim}(M)_{[5:5]}=0$ as expected from \Cref{lem::dimvectorofM_linear}.

\end{example}

\begin{proof}[Proof of Theorem \ref{thm::h=Coxeter_perm}]
For a linear Nakayama algebra the Cartan matrix $\omega_A$ is lower triangular, so by Lemma \ref{bruhatfromcartan}, we can read off the Coxeter permutation $p_A$ from a Bruhat factorisation of the Cartan matrix.
More precisely, $p_A:\{1,\ldots,n\}\rightarrow\{1,\ldots,n\}$ is the row-permutation associated to $P[\omega_A].$

Our goal is to show 
\begin{equation}
\label{eq::pi_smallereq_h}
    p_A(i)\ge \hat{h}(i)
\end{equation} for every vertex $i.$ Then, since both $p_A$ and $\hat{h}$ are permutations on the set $\{1,\ldots,n\},$ they must coincide. To calculate $p_A$, we will use the second formula from Proposition \ref{coxeterpermformula} for the matrix $\omega_A$. More precisely, we will prove that for every $i\in\{1,\ldots,n\}$ we have
\begin{equation}
\label{eq:span}
    r_i(\omega_A,\hat{h}(i)-1)\in \langle  r_{i+1}(\omega_A,\hat{h}(i)-1), r_{i+2}(\omega_A,\hat{h}(i)-1), \cdots , r_n(\omega_A,\hat{h}(i)-1) \rangle_K .
\end{equation}
 This implies (\ref{eq::pi_smallereq_h}). 

Throughout this proof, we will heavily rely on the fact that the $j$-th row of $\omega_A$ coincides with the dimension vector of the indecomposable injective $I(j),$ see for example \cite[Chapter III, Proposition 3.8]{ASS}.

It is easy to see (\ref{eq:span}) if $I(i)$ is projective-injective. Indeed, if $I(i)$ is projective, $e(S_i)=0$ and $S_{\hat{h}(i)}\cong \top I(i).$ Then for every $j\ge i$, $I(j)$ has no composition factor $S_k$ with $k\le \hat{h}(i)-1.$ Thus, $r_j(\omega_A,\hat{h}(i)-1)=0$ for every $j\ge i,$ which proves the statement. 
 
 Another special case with which we would like to deal separately is the case when $I(i)$ is not projective but $\pdim S_i=1.$ In this case, $e(S_i)=1$, $\Omega(S_i)\cong P_{i+1}$ and $\hat{h}(i)=i+1.$ Moreover, as $I(i)$ is not projective, $(\underline{\dim}\  I(i))_j=(\underline{\dim}\  I(i+1))_j$ for $j\le i=\hat{h}(i)-1.$ Thus, $r_i(\omega_A,\hat{h}(i)-1)=r_{i+1}(\omega_A,\hat{h}(i)-1)$ and (\ref{eq:span}) holds in this case as well.

 From now on, we assume that $I(i)$ is not projective and $\pdim S_i>1$. Note that this is equivalent to $e(S_i)\ge 2$ by \Cref{lem::e(S)odd_or_even}. Thus, we can apply Lemma \ref{lem::existanceOfM} and \ref{lem::dimvectorofM_linear} to obtain an indecomposable $A$-module $M$ together with a minimal injective coresolution 
 \begin{equation}
    0\rightarrow M \rightarrow I^0  \rightarrow I^1 \rightarrow \ldots \rightarrow I^d \rightarrow 0
\end{equation}
satisfying $\underline{\dim}(M)=r_{j}(\omega_A,\hat{h}(i)-1)$ for some $j>i$ and  $I^d\cong I(i)$. 

Since a minimal injective coresolution is exact, we can express $\underline{\dim}(I^d)$ as a linear combination of $\underline{\dim}(I^{0}), \  \underline{\dim}(I^{1}),\  \ldots,\ \underline{\dim}(I^{d-1})$ and $\underline{\dim}(M)$. The same is true if we restrict these dimension vectors only to their entries in places $[1,2,\ldots, \hat{h}(i)-1].$
Note that $I^d\cong I(i)$ implies that every other term $I^r$ with $r<d$ is of the form $I^r\cong I(\ell)$ for some $n\ge \ell>i$.
This follows, for example, from the description of injective resolutions in Nakayama algebras, see (\ref{eq::injResolUsingg}).
 Using the fact that the $\ell$-th row of $\omega_A$ equals $\underline{\dim}(I(\ell)),$ and that $\underline{\dim}(M)_{[1:\hat{h}(i)-1]}=r_j(\omega_A,\hat{h}(i)-1)$ for some $j>i$,  
 we obtain a formula for $r_i(\omega_A,\hat{h}(i)-1)$ as a linear combination of the vectors $r_{i+1}(\omega_A,\hat{h}(i)-1),\ldots,r_n(\omega_A,\hat{h}(i)-1)$.
\end{proof}

The next example shows that in the cyclic case, there is in general no linear order that is a cyclic shift of the order $1<2<3<4<5$ so that Theorem \ref{thm::h=Coxeter_perm} holds:
\begin{example}
Let $A=KQ/I$ be the cyclic Nakayama algebra with quiver $Q$ given by 
\[\begin{tikzcd}
	1 & 2 \\
	4 & 3
	\arrow["{a_1}", from=1-1, to=1-2]
	\arrow["{a_2}", from=1-2, to=2-2]
	\arrow["{a_4}", from=2-1, to=1-1]
	\arrow["{a_3}", from=2-2, to=2-1]
\end{tikzcd}\]
and relations $I=\langle a_1 a_2, a_3 a_4 a_1 \rangle.$
In this example, Ringel's homological permutation is given by
\[
\left(
\begin{array}{cccc}
1 & 2 & 3 & 4 \\
2 & 4 & 3 & 1
\end{array}
\right)
\]
Calculation gives that there is no cyclic shift of the ordering of the vertices for which the Coxeter permutation coincides with Ringel's homological permutation.

\end{example}

Our methods in the proof of \Cref{lem::existanceOfM} also allow us to show the following statement, which was posed as a question by Ringel in \cite[Remark 3.9]{R}. Note that this theorem holds for every Nakayama algebra, not only for the linear ones. 

\begin{theorem}
\label{thm::del(S)=e(S)}
Let $A$ be a Nakayama algebra and $S$ a simple $A$-module. Then $\del(S)=e(S)$.

\end{theorem}
\begin{proof}
We will use that by \cite[3.9]{R} we always have $\del(S)\le e(S)$ and $\del(S)=0$ if and only if $S$ is a submodule of a projective module, which follows directly by the definition of the delooping level.
We will first treat the cases $e(S)=0$ and $e(S)=1$ extra.
Assume first that $e(S)=0$, which by Lemma \ref{lem::e(S)odd_or_even} means that $0=e(S)=\pdim I(S)$. Thus the injective envelope $I(S)$ of $S$ is projective and thus $S$ is a submodule of a projective module and therefore $\del(S)=0$.
Now assume $e(S)=1$, which by Lemma \ref{lem::e(S)odd_or_even} means that $1=e(S)=\pdim S \leq \pdim I(S)$. Thus $I(S)$ is not projective. Now assume $\del(S)=0$ so that $S$ would embed into a projective module $P$. Since $S$ is simple, it also embeds into an indecomposable projective module $P'$, which must have the same socle as $S$, namely $S$ itself, since $A$ is a Nakayama algebra. 
But then $P'$ and $S$ have the same injective envelope $I(S)$, which is not projective. This is a contradiction, since Nakayama algebras are 1-Gorenstein and thus the injective envelope of every indecomposable projective module is projective. Thus we must have $\del(S) \geq 1$ and the inequality $\del(S) \leq e(S)=1$ forces $\del(S)=1$. Now assume $e(S) \geq 2$ in the following.

We would like to construct an $A$-module $N$ such that $\idim N =e(S)$ and $I(S)$ is the last term in a minimal injective coresolution of $N.$ This would ensure that $\Ext^{e(S)}(S,N)\neq 0$ by Lemma \ref{Bensonlemma}. Then we can apply Lemma \ref{lem:dell_lowerbound} to obtain $\del(S)\ge e(S).$ We have  $\del(S)\le e(S)$ for every simple module $S$, so this would conclude the proof for $\del(S)=e(S)$.

Recall that in \Cref{lem::existanceOfM} we constructed a module $M$  such that $\idim M=e(S)-1$ and $I(S)$ is the last term in a minimal injective coresolution of $M.$ Thus, if we choose $N$ so that the first cosyzygy module of $N$ is $M,$ i.e. $\Omega^{-1}(N)\cong M$, then $N$ has the desired properties. 

Since this $M$ was constructed as a factor module of an indecomposable injective module $I(j)$, we can choose $N$ to be the kernel of the projection $p: I(j)\twoheadrightarrow  M$. As both $I(j)$ and $M$ are uniserial, $p$ is uniquely determined up to isomorphism.  Uniseriality of $I(j)$ also implies that the canonical inclusion $N=\ker (p) \hookrightarrow I(j)$ is a minimal injective envelope of $N$, thus $N\cong \Omega^{-1}(M)$, which concludes our proof.

\end{proof}

The next example shows that the result is not true for general indecomposable modules $M$ instead of simple modules when defining $e(M):=\min \{ \pdim M, \pdim I(M) \}$.

\begin{example}
Let $A$ be the Nakayama algebra with Kupisch series [2,3,3,2,2,1] and simple modules numbered from 1 to 6. Let $M=e_2A/e_2 J^2$, which is a two-dimensional $A$-module. Then $M$ has delooping level 2, but $e(M)=3.$
\end{example}

\newpage

\section{Auslander regularity and the Coxeter matrix}
\label{sec::auslander_regularity_and_Coxeter_matrix}

We give the following definition:
\begin{definition}
\label{def::weel-def_ARmap}
Let $A$ be an algebra of finite global dimension.
We say that $A$ has a \textcolor{blue}{\emph{well-defined inverse Auslander-Reiten map}}  if the last non-zero term in a minimal injective coresolution of every indecomposable projective module is indecomposable.
\end{definition}
Let $A$ be an algebra with a well-defined inverse Auslander-Reiten map and assume that its underlying quiver has vertices $\{1,\ldots,n\}.$ Define the inverse Auslander-Reiten map of $A$ $\textcolor{blue}{\sigma}: \{1,\ldots,n\}\rightarrow \{1,\ldots,n\}$ to be the map  $\sigma(x)=y$ if $I(y)\cong \Omega^{-\idim P(x)} P(x).$ Any linear Nakayama algebra has a well-defined inverse Auslander-Reiten map as every indecomposable module has finite minimal injective coresolution, and all the cosyzygies in the resolution are indecomposable. 
Also every Auslander regular algebra has a well-defined inverse Auslander-Reiten map
 and in this case $\sigma=\hat{\psi}^{-1}$, the inverse of the Auslander-Reiten permutation, see Section \ref{subsec::permutations_on_simples}.
We briefly recall some basics on the Euler characteristic for finite dimensional algebras of finite global dimension $A$, see \cite[Chapter 3 III.]{ASS}.
The \textcolor{blue}{\emph{Euler characteristic}} of $A$ with $n$ simple modules and Cartan matrix $\omega_A$ is defined as 
the $\mathbb{Z}$-bilinear form 
$$\langle -, - \rangle_A: \mathbb{Z}^n \times \mathbb{Z}^n \rightarrow \mathbb{Z}$$
given by $\langle x, y \rangle_A= x^T (\omega_A^{-1})^T y$.
A quiver is acyclic if it does not contain oriented cycles.  
Note that an acyclic quiver algebra $KQ/I$ has finite global dimension, thus, there exists a Coxeter matrix associated to it.
We say that the vertices (equivalently, the simples) of an acyclic quiver algebra $KQ/I$ are \textcolor{blue}{\textit{naturally labelled}} or \textcolor{blue}{\textit{naturally ordered}} if the vertices are totally ordered such that $i<j \in Q_0$ whenever there is an oriented path from $i$ to $j$ and $i\neq j.$ We remark that the notion of being naturally labelled is used in combinatorics, see for example \cite{BCK}, and we extend this notion here from Hasse quivers of finite posets to general acyclic quivers.
If $Q$ is acyclic, such an ordering always exists.
This generalises the canonical labelling (\ref{eq:canonical_labelling}) on linear Nakayama algebras.

\begin{proposition}
\label{prop::Euler_characteristic_formula}
Let $A$ be a finite dimensional algebra of finite global dimension $g$.
If $M$ and $N$ are $A$-modules with dimension vectors $x$ and $y$, then we have 
that $\langle x, y \rangle_A= \sum\limits_{j=0}^{g}{(-1)^j \dim \Ext_A^j(M,N)}$
\end{proposition}
For a proof, see for example \cite[Chapter 3, Proposition 3.13]{ASS}.

For the purpose of the next theorem, we also need the following definitions:
\begin{definition}
Let $A$ be a naturally labelled acyclic algebra.
An $A$-module $M$ \textcolor{blue}{\textit{satisfies property $\circledast$}} if the following property is satisfied:
Let $0 \rightarrow M \rightarrow I^0 \rightarrow \cdots \rightarrow I^n \rightarrow 0 $ be a minimal injective coresolution of $M$, then $I^n=I(x)$ is an indecomposable module and $x <y$ for all $y$ such that $I(y)$ is a direct summand of $I^i$ for some $0 \leq i <n$.
We say that an algebra $A$ \textcolor{blue}{\textit{satisfies property $\circledast$}} if every indecomposable projective $A$-module satisfies property $\circledast$.
\end{definition}

We give an example:
\begin{example}
Let $A=KQ/I$ be the algebra with quiver $Q=$
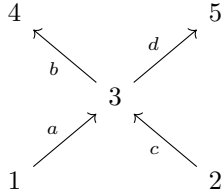
\[\begin{tikzcd}
	4 && 5 \\
	& 3 \\
	1 && 2
	\arrow["b", from=2-2, to=1-1]
	\arrow["d", from=2-2, to=1-3]
	\arrow["a", from=3-1, to=2-2]
	\arrow["c", from=3-3, to=2-2]
\end{tikzcd}\]
and relations $I=\langle ab, cd \rangle.$
This algebra is acyclic and its vertices are naturally labelled.
We have the following minimal injective coresolutions of the indecomposable projective $A$-modules:
$$0 \rightarrow P(1) \rightarrow I(5) \rightarrow 0$$
$$0 \rightarrow P(2) \rightarrow I(4) \rightarrow 0$$
$$0 \rightarrow P(3) \rightarrow I(4) \oplus I(5) \rightarrow I(3) \rightarrow 0$$
$$0 \rightarrow P(4) \rightarrow I(4) \rightarrow I(3) \rightarrow I(1) \rightarrow 0$$
$$0 \rightarrow P(5) \rightarrow I(5) \rightarrow I(3) \rightarrow I(2) \rightarrow 0$$
We have $\pdim I(4)=\pdim I(5)=0$, $\pdim I(1)= \pdim I(2)=2$ and $\pdim I(3)=1$.
Thus $A$ is Auslander regular, satisfies property $\circledast$ and the inverse Auslander-Reiten map is a bijection with $\sigma(1)=5, \sigma(2)=4, \sigma(3)=3, \sigma(4)=1$ and $\sigma(5)=2$.
\end{example}

\begin{theorem}
\label{thm::C=PU_iff_ARbij}
    Let $A$ be an acyclic naturally labelled algebra with a well-defined inverse Auslander-Reiten map that satisfies property $\circledast$. Then the following are equivalent:
    \begin{enumerate}[\rm (i)]
        \item The inverse Auslander-Reiten map $\sigma$ is a bijection.
        \item There exists a Bruhat decomposition of the Coxeter matrix  $C_A=U_1PU_2$ where $U_1=Id.$
    \end{enumerate}
In this case, the Coxeter permutation coincides with the inverse of $\sigma$.
    
\end{theorem}

\begin{proof} 
Note that (ii) can also be expressed as the matrix $C_A$ being upper triangular up to reordering its rows. 
Let $C_A=(C_{ij})_{i,j}$ and define $$p(i)=\min \{j\ | \ C_{ij}\neq 0 \}$$ for a vertex $i\in \{1,2,\ldots,n\}.$ 
Then $C_A$  is upper triangular up to reordering the rows if and only if $p$ is a bijection. 
Note that for the "only if" part, we use that $C_A$ is invertible. 

    
    
\begin{claim*}
        We can describe the entries of $C_A$ as $$C_{ij}=-\langle \underline{\dim}(S_j),\underline{\dim}(P(i))\rangle_A.$$
    \end{claim*}
\begin{proof}
We know that $C_A=- \omega_A^T \omega_A^{-1}$.
We have $C_{ij}=e_iT C_A e_j = - e_i^T\ \omega_A^T \omega_A^{-1} e_j$. 
Now we can apply the transpose to this to get $C_{ij}$ since this is a scalar and the transpose is the same number:
$C_{ij}=- (e_i^T \ \omega_A^T \omega_A^{-1} e_j)^T=-(e_j^T) (\omega_A^{-1})^T (\omega_A e_i).$ 
Recall that the $i$-th column of the Cartan matrix of $A$ equals the dimension vector of $P(i).$ 
Thus, by definition of the Euler characteristic, we have $C_{ij}=- \langle \underline{\dim} (S_j) , \underline{\dim}  (P(i))  \rangle_A$.
\end{proof}
    
    
    \begin{claim*}
        The maps $p$ and $\sigma$ coincide.
    \end{claim*}
    \begin{proof}
        We fix an $i\in \{1,\dots,n\}$ and consider a minimal injective coresolution of $P(i)$:
        \[0\rightarrow P(i) \rightarrow I^0\rightarrow I^1 \rightarrow \ldots \rightarrow I^d\rightarrow 0. \]
       
        Since $I^d\cong I(\sigma(i))$ by the definition of $\sigma,$ $C_{i\sigma(i)}=-\langle {\underline{\dim}}(S_{\sigma(i)}),{\underline{\dim}}(P(i))\rangle_A = \pm \dim \Ext_A^d(S_{\sigma(i)},P(i)) \neq 0$. This follows from Proposition \ref{prop::Euler_characteristic_formula} and the fact that $I(\sigma(i))$ appears only in $I^d$ as a direct summand since $A$ has property $\circledast$. Thus, $p(i)\le \sigma(i).$
        Property $\circledast$ of the indecomposable projectives also implies $S_j\ncong \soc(I^r)$ for any $j<\sigma(i)$ and any $0\le r \le d.$ 
 
        Here we used that $S_{\sigma(i)}\cong \soc(I^d)$.
 This shows that $C_{ij}=-\langle {\underline{\dim}}(S_j),{\underline{\dim}}(P(i))\rangle_A = 0$ for every $j<\sigma(i),$ which proves that $p(i)=\sigma(i).$
    \end{proof}

Thus, $\sigma$ is a bijection if and only if the Coxeter matrix is upper triangular up to reordering its rows by a permutation. 
For the last statement of the theorem, note that if $p$ is bijective then it coincides with the inverse of the Coxeter permutation.
\end{proof}

We have the following corollary, that we state as a theorem due to its importance:
\begin{theorem} \label{maintheoremauslanderregimpliescoxbruhat}
Let $A$ be a naturally labelled acyclic Auslander regular algebra.
Then there exists a Bruhat decomposition of the Coxeter matrix $C_A=U_1 P U_2$ with $U_1=\id$.
In this case, the Coxeter permutation coincides with the Auslander-Reiten permutation.
\end{theorem}
\begin{proof}
Let 
$$0 \rightarrow P \rightarrow I^0 \xrightarrow{f^0} I^1 \xrightarrow{f^1} \cdots \xrightarrow{f^{n-2}} I^{n-1} \xrightarrow{f^{n-1}} I^n \rightarrow 0$$
be a minimal injective coresolution of the indecomposable projective module $P$.
First note that for an Auslander regular algebra, the last term $I^n$ is always indecomposable, thus $A$ has a well-defined Auslander map $\sigma$. Moreover, 
$I^n$ has projective dimension equal to $n$ and $\sigma$ is bijective. All these statements follow from \cite[Proposition 5.4]{AR}.
Therefore, this theorem is a direct consequence of Theorem \ref{thm::C=PU_iff_ARbij} once we have shown that naturally labelled acyclic Auslander regular algebras satisfy property $\circledast$.
We claim that the maps $f^i: I^i \rightarrow I^{i+1}$ have the property that $f^i(X) \neq 0$ for every indecomposable direct summand $X$ of $I^i$.
By \cite[Proposition 5.4]{AR}, every non-zero cosyzygy $\Omega^{-i}(P)$ of an indecomposable projective $A$-module $P$ is indecomposable in an Auslander regular algebra $A$.
Now assume $X$ is an indecomposable direct summand of $I^i$ with $f(X)=0$ for some $i<n$.
Note that $X$ is also injective as a direct summand of the injective module $I^i$.
We have the following diagram:
\[\begin{tikzcd}
	0 & {\ker f^i} & {I^i} & {I^{i+1}} \\
	&& X \\
	&& 0
	\arrow[from=1-1, to=1-2]
	\arrow["{i_1}", from=1-2, to=1-3]
	\arrow["{f^i}", from=1-3, to=1-4]
	\arrow["{\exists g}", from=2-3, to=1-2]
	\arrow["{i_2}"', from=2-3, to=1-3]
	\arrow[from=3-3, to=2-3]
\end{tikzcd}\]
Here $i_2 : X \rightarrow I^i$ is the canonical inclusion of $X$ as a direct summand into $I^i$ and $i_1$ is the inclusion of the kernel of $f^i$ into $I^i$. By the universal property of the kernel, there exists a unique map $g: X \rightarrow \ker f^i$ with $i_1 g=i_2$ and $g$ is necessarily injective. But $\ker f^i$ is indecomposable as a cosyzygy of an indecomposable projective module and $X$ is injective, thus, $g$ must split and $X$ must be isomorphic to $\ker f^i$. Thus, $i_1$ needs to split, which is impossible since the coresolution is assumed to be minimal.
This forces $f^i(X) \neq 0$, as claimed.
So if $I(a)$ is an indecomposable injective direct summand of $I^i$, then there exists an indecomposable injective direct summand $I(b)$ of $I^{i+1}$ with $\Hom_A(I(a),I(b)) \neq 0$, which implies $a \geq b$ in the ordering of the vertices using that $A$ has an acyclic quiver. But since $A$ is assumed to be Auslander regular, we have $\pdim I^i<n$ for all $i<n$ and since $\pdim I^n=n$, the indecomposable injective module $I^n$ can not appear as a direct summand in some earlier terms $I^i$. This shows that $A$ satisfies property $\circledast$.
For the last statement of the theorem, use the last statement of Theorem \ref{thm::C=PU_iff_ARbij} and that for Auslander regular algebras $\sigma$ is the inverse of the Auslander-Reiten permutation.
\end{proof}

Following \cite[Definition 1.1]{DJMSSTW}, we call a finite dimensional acyclic algebra \textcolor{blue}{\emph{Coxeter-independent}} if the Coxeter permutation does not depend on the natural labelling. We remark that in \cite{DJMSSTW}, this was defined for incidence algebras of posets and there the term echelon-independent instead of Coxeter-independent was used. In \cite[Conjecture 8.2]{DJMSSTW} it was conjectured that Auslander regular incidence algebras of posets are Coxeter-independent. The next theorem proves this conjecture for the much more general class of acyclic Auslander regular algebras.
\begin{theorem}
\label{cor::coxeter-indep}
Any acyclic Auslander regular algebra is Coxeter-independent and the Coxeter permutation coincides with the Auslander-Reiten permutation.
\end{theorem}
\begin{proof}
This is a direct consequence of Theorem \ref{maintheoremauslanderregimpliescoxbruhat}.
\end{proof}

Specialising to posets, we obtain the following corollary, which answers one direction of \cite[Question 5.3]{KMT}: 
\begin{corollary}
Let $P$ be a bounded poset whose incidence algebra $KP$ is naturally labelled and Auslander regular.
Then the Coxeter matrix of $KP$ has a Bruhat factorisation $C=U_1 P U_2$ with $U_1$ the identity matrix.
\end{corollary}
Next we apply Theorem \ref{thm::C=PU_iff_ARbij} to linear Nakayama algebras and then 2-Gorenstein monomial algebras.
\begin{theorem}
\label{thm::C=PU_iff_ARbijnak}
    Let $A$ be a linear Nakayama algebra with canonical ordering of the simple modules, as in (\ref{eq:canonical_labelling}). Then the following are equivalent:
    \begin{enumerate}[\rm (i)]
        \item $A$ is Auslander regular.
        \item $\sigma$ is a bijection.
        \item There exists a Bruhat decomposition of the Coxeter matrix  $C_A=U_1PU_2$ where $U_1=Id.$
    \end{enumerate}
In this case, the Coxeter permutation coincides with the Auslander-Reiten permutation.
\end{theorem}

\begin{proof}
As discussed after Definition \ref{def::weel-def_ARmap}, all linear Nakayama algebras have a well-defined inverse Auslander-Reiten map. 
So by Theorem \ref{thm::AGcharacterisationviaARbij_monomial} and \Cref{rmk::inverse_ARbij}, (i) and (ii) are equivalent.
To show that (ii) and (iii) are equivalent we just have to note that linear Nakayama algebras satisfy property $\circledast$ by Theorem \ref{thm::C=PU_iff_ARbij}. This is clear from the shape of minimal injective coresolutions in linear Nakayama algebras, as described in Section \ref{subsec::Nakayama_notation} of the preliminaries.
\end{proof}

We remark that the statement of Theorem \ref{thm::C=PU_iff_ARbij} does not extend to cyclic Nakayama algebras in an obvious way:
\begin{example}
\label{ex::cyclic_AG_but_not_PU}
Let $A=KQ/I$ be the cyclic Nakayama algebra with quiver $Q$ given by 
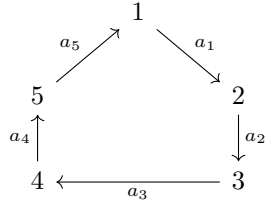
\[\begin{tikzcd}
	& 1 \\
	5 && 2 \\
	4 && 3
	\arrow["{a_1}", from=1-2, to=2-3]
	\arrow["{a_5}", from=2-1, to=1-2]
	\arrow["{a_2}", from=2-3, to=3-3]
	\arrow["{a_4}", from=3-1, to=2-1]
	\arrow["{a_3}", from=3-3, to=3-1]
\end{tikzcd}\]
and relations $I=\langle a_1 a_2, a_2 a_3, a_4 a_5  \rangle$.

Then $A$ is Auslander regular, but there is no order on the vertices of $Q$ that comes from a cyclic shift of the order $1<2<3<4<5$, such that the Coxeter matrix $C_A$ is upper triangular up to permuting its rows. 
Ringel's homological is bijection given by 
\[
\left(
\begin{array}{ccccc}
1 & 2 & 3 & 4 & 5 \\
4 & 5 & 2 & 1 & 3
\end{array}
\right)
\]

\end{example}
While the results do not carry over to the general cyclic case, we will in future work show that some of our results in this article work at least for a nice subclass of cyclic Nakayama algebras.

Next, we extend our results from linear Nakayama algebras to 2-Gorenstein acyclic monomial algebras:
\begin{theorem}
\label{thm::AG_iff_C=PU_acyclic_2Gorenstein_monomial}
Let $A$ be a 2-Gorenstein, acyclic monomial algebra with a natural labelling on its simple modules. Then the following are equivalent:
\begin{enumerate}[\rm (i)]
    \item A is Auslander regular.
    \item There exists a Bruhat decomposition of the Coxeter matrix $C_A=U_1PU_2$ where $U_1=Id.$
\end{enumerate}
In this case, the Coxeter permutation coincides with the Auslander-Reiten permutation.
\end{theorem}

\begin{proof}
    Assume that the underlying quiver of $A$ has $n$ vertices labelled $\{1,\ldots,n\}.$ We will first show that $A$ satisfies property $\circledast$.
First, $\Omega^{- \idim P(x)}P(x)$ is always indecomposable by Lemma \ref{lem:InjResOfProj_2gorensteinmonomial}.

    Second, if $$0\to P(i)\to I^0 \to I^1 \to \ldots \to I^d \to 0$$ is a minimal injective coresolution and $I(x)$ is a direct summand of $I^j$ and $I(y)$ is a direct summand of $I^k$ for some $0\le j<k\le d$ then $x>y$ by Corollary \ref{cor:decreasing_labels_injective_resol} and the definition of a natural labelling on the simple $A$-modules.
     In conclusion, the proof of Theorem \ref{thm::C=PU_iff_ARbij} also works for $2$-Gorenstein, acyclic monomial algebras. 
     
    Now the result follows from Theorem \ref{thm::AGcharacterisationviaARbij_monomial}, \Cref{rmk::inverse_ARbij} and Theorem \ref{thm::C=PU_iff_ARbij}
\end{proof}

\begin{remark}
    Note that the statement of \Cref{thm::AG_iff_C=PU_acyclic_2Gorenstein_monomial} fails for general acyclic monomial algebras. 
    For example, \cite[Example 5.2]{KMT} illustrates an acyclic monomial algebra with a natural labelling on its simple modules that is not Auslander regular but whose Coxeter matrix is upper triangular up to reordering its rows. 
    This example is not $2$-Gorenstein.
    
\end{remark}

\section{Ringel's bijection for Auslander-Gorenstein algebras}

In this section, we prove that Iyama's grade bijection coincides with Ringel's bijection for Nakayama algebras that are Auslander-Gorenstein. We should note, that the previous sections only imply this for the linear case.
We will use the following simplified version of a proposition by Madsen \cite[Proposition 2.2]{Mad}:
\begin{proposition} \label{Madsenlemma}
Let $A$ be a Nakayama algebra and $M$ and $N$ indecomposable $A$-modules such that $M$ is a submodule of $N$.
If $N$ has odd projective dimension, then $M$ has odd projective dimension as well, and $\pdim M \leq \pdim N$.
\end{proposition}

Recall that a module $M$ is called \textcolor{blue}{\emph{perfect}} if $\grade M=\pdim M$. Perfect modules play an important role in commutative algebra, see for example \cite{BH}, and also in non-commutative algebra, for example in the study of D-modules \cite{Bj}.
Next, we show that every simple module $S$ whose injective envelope $I(S)$ has odd projective dimension, enjoys strong homological properties:
\begin{proposition}\label{oddSimplePerfect}
Let $A$ be an Auslander-Gorenstein Nakayama algebra. Then for every simple $A$-module $S$ for which $I(S)$ has odd projective dimension, we have $\pdim I(S)=\pdim S$ and $S$ is perfect. 
\end{proposition}
\begin{proof}
    By Proposition \ref{Madsenlemma}, we have that $\pdim S \le \pdim I(S)$ as $S$ is a submodule of $I(S)$ and $\pdim I(S)$ is odd by assumption. 
    We always have $\grade(S)\le \pdim S ,$ since $\pdim S < \infty$. 
    This is immediate if we consider the definition of the grade of a module and use the following characterisation of the projective dimension of a module with finite projective dimension from \cite[Lemma VI.5.5]{ARS}:
    $\pdim M=\sup\ \{n\in\mathbb{N}_0 \ | \ \Ext^n_A(M,A)\neq 0\}$ if $\pdim M<\infty$.
    Now Theorem \ref{AGcondition} tells us that $\grade(S)=\pdim I(S)$ since $A$ is Auslander-Gorenstein.
    Thus 
    $$\pdim I(S) = \grade S \leq \pdim S \leq \pdim I(S)$$ 
    forces that we have 
    $$\pdim I(S)=\pdim S = \grade S .$$
\end{proof}

For the proof of the next theorem recall that  $e(S)=\min \{ \pdim S, \pdim I(S) \}$ and Ringel's homological bijection is defined as $h(S)=\top \Omega^{e(s)}(N(S)),$ 
where $N(S)=S$ if $e(S)$ is odd and $N(S)=I(S)$ if $e(S)$ is even. Recall that $\phi(S)=\top D \Ext_A^{g_S}(S,A)$ denotes Iyama's grade bijection with $g_S=\grade S.$
\begin{theorem} \label{RingelsPerm=ARPerm}
    Let  $A$ be an Auslander-Gorenstein Nakayama algebra. Then for every simple $A$-module $S$ we have $\grade(S)=e(S)$ and 
    $$h(S)=\phi(S).$$
\end{theorem}

\begin{proof}
    First, consider the case when $e(S)$ is even. Then by definition, $N(S)=I(S).$ Moreover, $e(S)=\pdim I(S),$ since by Theorem \ref{AGcondition} $\pdim I(S) =\grade(S),$ and $\grade(S)\le \pdim S$ as already discussed in the proof of Proposition \ref{oddSimplePerfect}. So in this case $\grade(S)=e(S)$ and $h(S)$ is equal to the top of the last term of the minimal projective resolution of $I(S).$ Thus, Theorem \ref{thm::gradebij=ARbij} gives us $h(S)=\phi(S)$ in this case. 

    Now assume $e(S)$ is odd. Then applying Proposition \ref{oddSimplePerfect}, we get that $\grade(S)=\pdim S=\pdim I(S) =e(S)$. Moreover, as discussed in \cite[Section 4.4. Remark]{R}, whenever $\pdim S=\pdim I(S)$ holds and $\pdim S$ is odd, we have $\Omega^{\pdim S}(S)=\Omega^{\pdim I(S)}I(S).$ In conclusion, we have in this case $h(S)=\top \Omega^{\pdim S}(S)=\top \Omega^{\pdim I(S)}I(S).$ Just as in the $e(S)$ even case, this implies $h(S)=\phi(S)$ here as well. 
\end{proof}
\begin{remark}
Note that because $\grade(S)=\pdim I(S)=e(S)$ in an Auslander-Gorenstein algebra by \Cref{AGcondition} (and similarly, $\cograde(S)=\idim P(S)=e^*(S))$, the property that $e(S)=e^*(h(S))$, which was found by Ringel, is equivalent to $\grade(S)=\cograde(\phi(S)),$ which is a general property of Auslander-Gorenstein algebras by results of Iyama \cite{I}.
\end{remark}

\begin{remark}
    We would like to emphasise that our result does not recover Ringel's bijection in its full generality. While $h$ defines a bijection for any Nakayama algebra, the grade bijection is only defined for the ones which are Auslander-Gorenstein.
\end{remark}

We have shown in \Cref{thm::del(S)=e(S)} that in a general Nakayama algebra, $e(S)= \del (S)$ for every simple module $S$ and for an Auslander-Gorenstein Nakayama algebra we have $e(S)=\grade(S)$, and thus, $\del(S)=\grade(S)$. The next proposition shows that the equality $\del(S)=\grade(S)$ for simple modules $S$ holds for general Auslander-Gorenstein algebras:

\begin{theorem}
\label{thm:delS=gradeS}
Let $A$ be an Auslander-Gorenstein algebra and $S$ a simple $A$-module. Then $\del (S)=\grade(S).$
\end{theorem}
\begin{proof}
Since for Auslander-Gorenstein algebras $\grade(S)=\pdim I(S)$ for every simple module $S$ by \cite[Theorem 3.3]{KMT}, it is enough to show $\del(S)=\pdim I(S)$. 

The direction $\del(S)\le \pdim I(S)$ follows from \cite[Proposition 2.3]{R} as $S$ is a submodule of $I(S).$ 

For the other direction, we use the statement from \cite[Proposition 5.4]{AR} which establishes the existence of the Auslander-Reiten bijection. 
More precisely, it tells us that in an Auslander-Gorenstein algebra, for every simple module $S$ there exists an indecomposable projective module $P$ so that $I(S)$ is isomorphic to the last term of the minimal injective coresolution of $P.$ 
In other words, $\psi^{-1}(P)=\Omega^{-\idim P}(P)\cong I(S).$ Moreover, $\idim P =\pdim I(S).$ Then we have $\Ext^{\idim P}(S,P)\neq 0$ by Lemma \ref{Bensonlemma}, so we can use Lemma \ref{lem:dell_lowerbound} to obtain the inequality $\pdim I(S)=\idim P\le \del(S).$
\end{proof}

In Proposition \ref{oddSimplePerfect}, we saw that every simple module $S$ for which $I(S)$ has odd projective dimension is perfect. 
There is a stronger property for simple modules $S$ called \textcolor{blue}{\emph{$n$-regularity}}, which is defined by the condition that $S$ satisfies $n=\pdim S=\grade (S)$ and $\dim \Ext_A^n(S,A)=1$.
This property plays an important role in commutative algebra, and especially for $n=2$ in the classification of exact structures. 
We refer, for example, to \cite{En} for more information. 
For Nakayama algebras, 2-regular simple modules were classified in \cite{MRS}.
We pose the following conjecture, which would imply even stronger homological properties for simple modules with odd grade in Auslander-Gorenstein Nakayama algebras:

\begin{conjecture}
Let $A$ be an Auslander-Gorenstein Nakayama algebra. Then $\dim \Ext_A^i(S,A) \leq 1$ for $S$ simple and odd $i$. 
\end{conjecture}

Combined with Proposition \ref{oddSimplePerfect}, the truth of this conjecture would imply that every simple module $S$ with odd grade $g$ is $g$-regular in an Auslander-Gorenstein Nakayama algebra.
In forthcoming work, we will study the classification of Auslander regular linear Nakayama algebras from a enumerative viewpoint and establish several classification results for some special cases.
We mention that the classification for special subclasses such as dominant Auslander regular algebras and higher Auslander algebras is a recent field of study, see for example \cite{R2}, \cite{Sh2} and \cite{STZ}.

\section*{Acknowledgement}
This project profited from the use of \cite{Sage} and the GAP-package \cite{QPA}.

\end{document}

%% file: inj.tex
\begin{tikzpicture}[line width=0.8pt, font=\LARGE] 
   
    \node (B) at  (4.5,1.3) {$I(t(q))=$};

    \node (A) at (9, 0) {$t(q)$};
    \node (D) at (8, 1) {};
    \node (E) at (10, 1) {};   
    \node (G) at (7, 2) {};
    \node (H) at (6, 3) {};
    \node (J) at (11, 2) {};
    \node (K) at (12, 3) {};
    \node (EJ) at (10.8,1.3) {\textcolor{purple}{$i_2$}};
    \node (DG) at (7.2,1.3) {\textcolor{cyan}{$i_1$}};

    \draw[->, color=purple] (E) to (A);
    \draw[->, color=cyan] (D) to (A);
    \draw[->, color=cyan] (G) to (D);
    \draw[->, color=cyan] (H) to (G);
    \draw[->, color=purple] (K) to (J);
    \draw[->,color=purple] (J) to (E);

     \draw [ decorate,decoration={brace,amplitude=5pt,mirror,raise=6pt},yshift=0pt]
     (K) -- (E) node [midway,xshift=-0.6cm,yshift=0.5cm] {\textcolor{black}{$ i_2^+$}};

    \draw [ decorate,decoration={brace,amplitude=5pt,mirror,raise=6pt},yshift=0pt]
     (A) -- (G) node [midway,xshift=0.5cm,yshift=0.5cm] 
     {$q$};

\end{tikzpicture}

%% file: inj_resol.tex
\begin{tikzpicture}[line width=0.8pt, font=\LARGE] 

    \node (A4) at (9, 0) {$i$};
    \node (D4) at (8, 1) {};
    \node (E4) at (10, 1) {};   
    \node (G4) at (7, 2) {};
    \node (H4) at (6, 3) {};
    \node (J4) at (11, 2) {};
    \node (K4) at (12, 3) {};


    \node (A1) at (-5, 0) {$i$};
    \node (B1) at (-6, -1) {};
    \node (C1) at (-4, -1) {};
    \node (I1) at (-7, -2) {$t_1$};
    \node (L1) at (-3, -2) {$t_2$};

    \node (A2) at (0, 0) {$i$};
    \node (B2) at (-1, -1) {};
    \node (C3) at (6, -1) {};
    \node (D3) at (4, 1) {};
    \node (E2) at (1, 1) {};
    \node (I2) at (-2, -2) {$t_1$};    
    \node (G3) at (3, 2) {};
    \node (J2) at (2, 2) {};
    \node (L3) at (7, -2) {$t_2$};
    \node (A3) at (5, 0) {$i$};

    \node (X) at (-8, 0) {$0$};
    \node (X01) at (-7, 0) {};
    \node (Y01) at (-6, 0) {};
    \node (X23) at (2.5, 0) {$\oplus$};
    \node (X12) at (-3.2, 0) {};
    \node (Y12) at (-1.7, 0) {};
    \node (X34) at (6.2, 0) {};
    \node (Y34) at (7.7, 0) {};
    \draw[->] (X01) to (Y01) {};
    \draw[->] (X12) to (Y12) {};
    \draw[->] (X34) to (Y34) {};
    
    \node (ZS) at (12.5, 0) {$0$};
    \node (ZS1) at (10.5, 0) {};
    \node (ZS2) at (11.5, 0) {};
    \draw[->] (ZS1) to (ZS2);

    \draw[->, purple] (E4) to (A4);
    \draw[->, cyan] (D4) to (A4);
    \draw[->, cyan] (G4) to (D4);
    \draw[->, purple] (J4) to (E4){};

    \draw[->, blue] (A1) to (B1);
    \draw[->, orange] (A1) to (C1);
    \draw[->, blue] (B1) to (I1)node[midway,xshift=-6.3 cm,yshift=-0.8cm ] {$ p_1$};
    \draw[->, orange] (C1) to (L1) node[midway,xshift=-3.7 cm,yshift=-0.8cm ] {$ p_2$};

    \draw[->, blue] (A2) to (B2);
    \draw[->, orange] (A3) to (C3);
    \draw[->, blue] (B2) to (I2) node[midway,xshift=-1.3 cm,yshift=-0.8cm ] {$ p_1$};
    \draw[->, orange] (C3) to (L3) node[midway,xshift=6.3 cm,yshift=-0.8cm ] {$ p_2$};
    \draw[->, purple] (J2) to (E2);
    \draw[->, purple] (E2) to (A2);
    \draw[->, cyan] (G3) to (D3);
    \draw[->, cyan] (D3) to (A3);

    \node at (0.7,1.3) {\textcolor{purple}{$ i_1$}}; 
    \node at (10.4,0.8) {\textcolor{purple}{$ i_1$}};
    \node at (4.3,1.3) {\textcolor{cyan}{$ i_2$}};
    \node at (7.6,0.8) {\textcolor{cyan}{$ i_2$}};

\end{tikzpicture}

%% file: fig_resolutions.tex
\begin{tikzpicture}[scale=0.5]

\def\xmax{33}
\def\ymax{29}


\draw[->, line width=1 pt] (32,14) -- (32,0.5);
\node at (31.5,13.5) {1};
\node at (31.5,12.5) {2};
\node at (31.5,10.5) {$i$};
\node at (31.2,9.5) {$i+1$};
\node at (31.5,1.5) {$n$};
\node at (31.5,11.5) {$\cdot$};
\node at (31.5,11.75) {$\cdot$};
\node at (31.5,11.25) {$\cdot$};

\node at (31.5,5.5) {$\cdot$};
\node at (31.5,5.75) {$\cdot$};
\node at (31.5,5.25) {$\cdot$};

\node at (32.5,13.5) {$\mathbb{Z}$};


\draw[->, line width=1 pt] (32,28.5) -- (32,14.5);
\node at (31.5,27.5) {1};
\node at (31.5,26.5) {2};
\node at (31.5,24.5) {$i$};
\node at (31.2,23.5) {$i+1$};
\node at (31.5,15.5) {$n$};
\node at (31.5,25.5) {$\cdot$};
\node at (31.5,25.75) {$\cdot$};
\node at (31.5,25.25) {$\cdot$};

\node at (31.5,19.5) {$\cdot$};
\node at (31.5,19.75) {$\cdot$};
\node at (31.5,19.25) {$\cdot$};

\node at (32.5,28) {$\mathbb{Z}$};





\fill[gray!30] (28,24) rectangle (29,25);
\node at (28.5,24.5) {$i$};
\node at (28.5,25.5) {$S_i$};

\fill[purple!30] (25,23) rectangle (26,25);
\node[purple] at (25.5,25.5) {$P_0$};
\fill[orange!30] (22,22) rectangle (23,24);
\node[orange] at (22.5,24.5) {$P_1$};
\fill[purple!30] (19,20) rectangle (20,23);
\node[purple] at (19.5,23.5) {$P_2$};
\fill[orange!30] (16,19) rectangle (17,22);
\fill[purple!30] (13,18) rectangle (14,20);
\node[purple] at (13.5,21) {$P_{e(S_i)-3}$};
\fill[orange!30] (10,17) rectangle (11,19);
\node[orange] at (10.5,20) {$P_{e(S_i)-2}$};
\fill[purple!30] (7,15) rectangle (8,18);
\node[purple] at (7.5,19) {$P_{e(S_i)-1}$};
\fill[orange!30] (4,15) rectangle (5,17);
\node[orange] at (4.5,17.6) {$P_{e(S_i)}$};


 \draw[ thick, ForestGreen] (8.5,17) rectangle (9.5,18.5);
\fill[ForestGreen!20, dotted] (8.5,17) rectangle (9.5,18.5);
\node at (9,16.5) {$M$};

\draw[ thick, ForestGreen] (11.5,17) rectangle (12.5,19.5);
 \fill[ForestGreen!20, dotted] (11.5,17) rectangle (12.5,19.5);
 \node at (12,16.5) {\textcolor{ForestGreen}{$I^0$}};

\node at (16.5,22.5) {$\ldots$};

\draw[ thick, ForestGreen] (17.5,19.5) rectangle (18.5,22.5);
\fill[ForestGreen!20, dotted] (17.5,19.5) rectangle (18.5,22.5);
\node at (18.7,19) {\textcolor{ForestGreen}{$I^{e(S_i)-5}$}};

\draw[ thick, ForestGreen] (23.5,22.5) rectangle (24.5,24);
\fill[ForestGreen!20, dotted] (23.5,22.5) rectangle (24.5,24);
\node at (24.5,22) {\textcolor{ForestGreen}{$I^{e(S_i)-3}$}};


\node at (1.5,16) {$0$};
\draw[->, line width=0.8 pt] (2.2,16) -- (3.5,16);
\draw[->, line width=0.8 pt] (5.2,16.5) -- (6.8,16.5);
\draw[->, line width=0.8 pt] (8.2,17.5) -- (9.8,17.5);
\draw[->, line width=0.8 pt] (11.2,18.5) -- (12.8,18.5);
\draw[->, line width=0.8 pt] (14.2,19.5) -- (15.8,19.5);
\draw[->, line width=0.8 pt] (17.2,21.5) -- (18.8,21.5);
\draw[->, line width=0.8 pt] (20.2,22.5) -- (21.8,22.5);
\draw[->, line width=0.8 pt] (23.2,23.5) -- (24.8,23.5);
\draw[->, line width=0.8 pt] (26.2,24.5) -- (27.8,24.5);


\fill[gray!30] (28,10) rectangle (29,12);
\node at (28.5,10.5) {$i$};
\node at (28.5,12.5) {$I(i)$};

\fill[blue!30] (25,9) rectangle (26,12);
\node[blue] at (25.5,12.5) {$R_0$};
\fill[orange!30] (22,8) rectangle (23,10);
\node[orange] at (22.5,10.5) {$R_1$};
\fill[blue!30] (19,7) rectangle (20,9);
\node[blue] at (19.5,9.5) {$R_2$};
\fill[orange!30] (16,5) rectangle (17,8);
\fill[blue!30] (13,4) rectangle (14,7);
\node[blue] at (13.5,8) {$R_{e(S_i)-3}$};
\fill[orange!30] (10,3) rectangle (11,5);
\node[orange] at (10.5,5.6) {$R_{e(S_i)-2}$};
\fill[blue!30] (7,2) rectangle (8,4);
\node[blue] at (7.5,5) {$R_{e(S_i)-1}$};
\fill[orange!30] (4,1) rectangle (5,3);
\node[orange] at (4.5,3.6) {$R_{e(S_i)}$};


\draw[ thick, ForestGreen ] (8.5,3) rectangle (9.5,4.5);
\fill[ForestGreen!20, dotted] (8.5,3) rectangle (9.5,4.5);
\node at (9,2.5) {$M$};

\draw[ thick, ForestGreen ] (14.5,4.5) rectangle (15.5,7.5);
\fill[ForestGreen!20, dotted] (14.5,4.5) rectangle (15.5,7.5);
\node at (15,4) {\textcolor{ForestGreen}{$I^1$}};

\node at (16.5,8.5) {$\ldots$};

\draw[ thick, ForestGreen] (20.5,7.5) rectangle (21.5,9.5);
\fill[ForestGreen!20, dotted] (20.5,7.5) rectangle (21.5,9.5);
\node at (21.5,7) {\textcolor{ForestGreen}{$I^{e(S_i)-4}$}};

\draw[ thick, ForestGreen] (26.5,9.5) rectangle (27.5,12);
\fill[ForestGreen!20, dotted] (26.5,9.5) rectangle (27.5,12);
\node at (27.5,9) {\textcolor{ForestGreen}{$I^{e(S_i)-2}$}};


\node at (1.5,1.5) {$\ldots$};
\draw[->, line width=0.8 pt] (2.2,2) -- (3.8,2);
\draw[->, line width=0.8 pt] (5.2,2.5) -- (6.8,2.5);
\draw[->, line width=0.8 pt] (8.2,3.5) -- (9.8,3.5);
\draw[->, line width=0.8 pt] (11.2,4.5) -- (12.8,4.5);
\draw[->, line width=0.8 pt] (14.2,6.5) -- (15.8,6.5);
\draw[->, line width=0.8 pt] (17.2,7.5) -- (18.8,7.5);
\draw[->, line width=0.8 pt] (20.2,8.5) -- (21.8,8.5);
\draw[->, line width=0.8 pt] (23.2,9.5) -- (24.8,9.5);
\draw[->, line width=0.8 pt] (26.2,11.5) -- (27.8,11.5);

\end{tikzpicture}

%% file: mainnew.bbl
\begin{thebibliography}{MTY}




\bibitem[AB]{AB} Alperin, J.~L.; Bell, R.~B.:{\it Groups and representations}, Graduate Texts in Mathematics, 162, Springer, New York, 1995; 





\bibitem[ASS]{ASS} Assem, I.; Simson, D.; Skowronski, A.: {\it Elements of the representation theory of associative algebras. Vol. 1: Techniques of representation theory. } London Mathematical Society Student Texts 65. Cambridge: Cambridge University Press, 458 p. (2006). 

\bibitem[ARS]{ARS} Auslander, M.; Reiten, I.; Smalo, S. O.: {\it Representation theory of Artin algebras.} Cambridge Studies in Advanced Mathematics, Volume 36, Cambridge University Press, 1997.  

\bibitem[AR]{AR} Auslander, M., Reiten, I.: {\it k-Gorenstein algebras and syzygy modules.} J. Pure Appl. Algebra, 92, 1-27 (1994).


\bibitem[B]{B} Bass, H.: {\it On the ubiquity of Gorenstein rings.} Math. Z. 82, 8-28 (1963).


\bibitem[Ben]{Ben} Benson, D.: {\it Representations and cohomology I: Basic representation theory of finite groups and associative algebras.} Cambridge Studies in Advanced Mathematics, Volume 30, Cambridge University Press, 1991.

\bibitem[BCK]{BCK} Bevan, D.; Cheon, G.; Kitaev, S: {\it On naturally labelled posets and permutations avoiding 12-34.} European Journal of CombinatoricsVolume 126, May 2025.

\bibitem[Bj]{Bj} Bj\"ork, J.: {\it Analytic D-modules and applications. }  Mathematics and its Applications (Dordrecht). 247. Dordrecht: Kluwer Academic Publishers. xiii, 581 p. (1993). 

\bibitem[BR]{BR} Butler, M. C. R.; Ringel, C.M.: \textit{Auslander-Reiten sequences with few middle terms and applications to string algebras.} Comm. Algebra, 15(1-2):145-179, 1987.

\bibitem[Bru]{Bru} Brualdi, R.: {\it Some combinatorially defined matrix classes. } Encinas, Andres M. (ed.) et al., Combinatorial matrix theory. Cham: Birkh\"auser. Adv. Courses in Math., CRM Barcelona, 1-46 (2018).

\bibitem[BH]{BH} Bruns, W.; Herzog, J.: {\it Cohen-Macaulay rings. Rev. ed.} Cambridge Studies in Advanced Mathematics. 39. Cambridge: Cambridge University Press. (1998). 

\bibitem[Bu]{Bu} Bump, D.: {\it Lie groups.} Graduate Texts in Mathematics 225. New York, NY: Springer. (2013). 

\bibitem[CB]{CB} Crawley-Boevey, W. W.: \textit{Maps between representations of zero-relation algebras.} J. Algebra, 126(2):259-263, 1989.


\bibitem[DJMSSTW]{DJMSSTW} Defant, C.; Jiang, Y.; Marczinzik, R.; Segovia, A.; Speyer, D. E.; Thomas, H.; Williams, N.: {\it Rowmotion and Echelonmotion.} https://arxiv.org/abs/2507.18230


\bibitem[En]{En} Enomoto, H.: {\it Classifications of exact structures and Cohen-Macaulay-finite algebras.} Advances in Mathematics
Volume 335, 7 September 2018, Pages 838-877.
 


\bibitem[G]{G} G\'elinas, V.: {\it The depth, the delooping level and the finitistic dimension.} Adv. Math. 394, 34 p. (2022). 

\bibitem[Gu]{Gu} Gustafson, W. H.: {\it Global dimension in serial rings.} J. Algebra 97, 14-16 (1985).



\bibitem[I]{I} Iyama, O.: {\it Symmetry and duality on $n$-Gorenstein rings}, Journal of Algebra, 269 (2003), no. 2, 528-535.



\bibitem[IM]{IM} Iyama, O.; Marczinzik, R.: {\it Distributive lattices and Auslander regular algebras.} Advances in Mathematics
Volume 398, 2022.

\bibitem[Kl{\'a}]{Kla} Kl\'asz, V.: {\it The Auslander-Gorenstein condition for monomial algebras.} https://arxiv.org/abs/2508.06957.

\bibitem[KM]{KM}
Kl\'asz, V.; Marczinzik, R.: {\it A survey on Auslander-Gorenstein algebras}, https://arxiv.org/abs/2508.19079

\bibitem[KMMRS]{KMMRS} Kl\'asz, V.; Marczinzik, R.; Mellit, A.; Rubey, M.; Stump, R.: {\it On the global dimension of Nakayama algebras.} https://arxiv.org/abs/2503.18415.

\bibitem[KMT]{KMT} Kl\'asz, V.; Marczinzik, R.; Thomas, H.: {\it Auslander regular algebras and Coxeter matrices.} https://arxiv.org/abs/2501.09447

\bibitem[KMM]{KMM} Ko, H.; Mazorchuk, V.; Mrden, R.: {\it Some homological properties of category $\mathcal{O}$. VI.} Doc. Math. 26, 1237-1269 (2021).

\bibitem[LP]{LP} Lenzing, H.; de la Pena, J. A.:{\it Spectral analysis of finite dimensional algebras and singularities.} Trends in representation theory of algebras and related topics. Proceedings of the 12th international conference on representations
of algebras and workshop. EMS Series of Congress Reports (2008), pp. 541-588.

\bibitem[Mad]{Mad} Madsen, D.: {\it Projective dimensions and Nakayama algebras.} Fields Institute Communications. 45. Amer. Math. Soc., Providence, RI, 2005. 247-265.

\bibitem[MRS]{MRS} Marczinzik, R.; Rubey, M.; Stump, C.: {\it A combinatorial classification of 2-regular simple modules for Nakayama algebras.} J. Pure Appl. Algebra 225, No. 3, 34 p. (2021). 




\bibitem[OOV]{OOV} Odeh, O. H.; Olesky, D. D.; van den Driessche, P.: {\it Bruhat decomposition and numerical stability.} SIAM J. Matrix Anal. Appl. 19, No. 1, 89-98 (1998).


\bibitem[QPA]{QPA} The QPA-team, QPA - Quivers, path algebras and representations - a GAP package, Version 1.33; 2022 https:
//folk.ntnu.no/oyvinso/QPA/. 

\bibitem[R]{R} Ringel, C. M.: {\it The finitistic dimension of a Nakayama algebra.} Journal of Algebra, Volume 576, 2021, Pages 95-145, ISSN 0021-8693.

\bibitem[R2]{R2} Ringel, C. M.: {\it Linear Nakayama algebras which are higher Auslander algebras.} Commun. Algebra 50, No. 11, 4842-4881 (2022). 

\bibitem[Sage]{Sage} A. Stein et al., Sage Mathematics Software, The Sage Development Team, 2022, http:
//www.sagemath.org.

\bibitem[STZ]{STZ} Sen, E.; Todorov, G.; Zhu, S.: {\it  Defect Invariant Nakayama Algebras.} https://arxiv.org/abs/2406.00254

\bibitem[Sh]{Sh}
Shen, D.: A note on homological properties of Nakayama algebras, Arch. Math. (Basel) {\bf 108} (2017), no.~3, 251--261; MR3614703

\bibitem[Sh2]{Sh2} Shen, D.: {\it  The resolution quiver of Nakayama algebras which are minimal Auslander-Gorenstein.} https://arxiv.org/abs/2511.14180







\bibitem[VO]{VO} Van Oystaeyen, F.: {\it Algebraic geometry for associative algebras.} Pure and Applied Mathematics 232. New York, NY: Marcel Dekker. vi, 287 p. (2000).


\end{thebibliography}
